\documentclass[final]{amsart}

\usepackage{epstopdf}
\usepackage{amstext,amssymb,amsmath,amsbsy,mathrsfs,mathtools,amsfonts,amscd,exscale}
\usepackage{graphics, epstopdf}
\usepackage{color}
\usepackage{amsmath, amssymb}
\usepackage{enumerate}
\usepackage{bigints}
\usepackage{verbatim}
\usepackage{cleveref}

\usepackage[notcite,notref,color]{showkeys}

\usepackage{graphicx}
\usepackage{epstopdf}
\usepackage{subcaption}
\usepackage{float}
\usepackage{placeins}
\usepackage{bm}
\usepackage{url}

\usepackage{algpseudocode,algorithm,algorithmicx}

\renewcommand{\sc}{\textsc}

\def\l{\left}
\def\r{\right}

\def\mZd{\mathbb{Z}^d}
\def\mR{\mathbb{R}}
\def\mRd{\mathbb{R}^d}
\def\mS{\mathbb{S}}
\def\St{\mathbb S_{\theta}}
\def\vt{v_{\theta}}
\def\mC{\mathcal{C}}
\def\mK{\mathcal{K}}
\def\Oc{\overline{\Omega}}
\def\Oh{\Omega_h}
\def\Ohc{\overline{\Omega}_h}
\def\Ohs{\Omega_{h,s}}
\def\Ohd{\Omega_{h,\delta}}
\def\Ohdd{\Omega_{h,d\delta}}
\def\pO{\partial\Omega}
\def\Th{\mathcal{T}_h}
\def\Vh{\mathbb{V}_h}
\def\Nh{\mathcal{N}_h}
\def\Nhi{\mathcal{N}_h^0}
\def\Nhb{\mathcal{N}_h^b}
\def\Nhv{\mathcal{N}_h^v}
\def\Nhe{\mathcal{N}_h^e}
\def\Nv{\mathcal{N}^v}
\def\Ne{\mathcal{N}^e}
\def\ve{\varepsilon}
\def\Tve{T_{\varepsilon}}
\def\uve{u_{\varepsilon}}
\def\MA{Monge-Amp\`{e}re }

\def\nCka{C^{k,\alpha}}
\def\nCkao{C^{k,\alpha}(\overline{\Omega})}
\def\Cka{C^{2+k,\alpha}}

\def\Wti{W^2_{\infty}}

\def\Czeroa{C^{0,\alpha}}
\def\Conea{C^{1,\alpha}}
\def\Czeroao{C^{0,\alpha}(\overline{\Omega})}

\def\interp{\mathcal I_h}
\def\lam1{\lambda_1}
\def\lamh1{\lambda_{1,\ve}}
\def\sdd{\nabla^2_{\delta}} 
\newcommand{\sd}[2]{\frac{\partial^2 #1}{\partial {#2}^2}}
\newcommand{\wt}[1]{\widetilde{#1}}

\def\ba{\bm{\alpha}}
\def\mA{\mathcal{A}}
\def\bu{\bm{u}}
\def\bw{\bm{w}}
\def\mRN{\mathbb{R}^N}
\newcommand{\argmax}{\textrm{arg\,max}}

\def\dist{\textrm{dist}}
\def\conv{\textrm{conv}}

\newtheorem{Theorem}{Theorem}[section]
\newtheorem{Lemma}[Theorem]{Lemma}
\newtheorem{Proposition}[Theorem]{Proposition}
\newtheorem{Corollary}[Theorem]{Corollary}

\theoremstyle{definition}

\newtheorem{example}{Example}
\numberwithin{example}{section}

\newtheorem{remark}[Theorem]{Remark}

\numberwithin{equation}{section}

\usepackage{amsthm}
\usepackage{thmtools}

\definecolor{red}{rgb}{1,0,0}
\definecolor{blue}{rgb}{0,0,1}
\definecolor{applegreen}{rgb}{0.55,0.71,0}

\begin{document}
\title[Two-Scale Methods for Convex Envelopes]{Two-Scale Methods for Convex Envelopes}

\author{Wenbo Li}
\address[Wenbo Li]{Department of Mathematics, University of Maryland, College Park, Maryland 20742}
\email[Wenbo Li]{wenboli@math.umd.edu}

\author{Ricardo H. Nochetto}
\address[Ricardo H. Nochetto]{Department of Mathematics, University of Maryland, College Park, Maryland 20742}
\email[Ricardo H. Nochetto]{rhn@math.umd.edu}
\thanks{Both authors were partially supported by the NSF Grant DMS -1411808. W. Li was also partially supported by the Patrick and Marguerite Sung Fellowship in Mathematics.}

\begin{abstract}
We develop two-scale methods for computing the convex envelope of 
a continuous function over a convex domain in any dimension.
This hinges on a fully nonlinear obstacle formulation \cite{Ob1}. We prove convergence and error estimates in the max norm. The proof utilizes a discrete comparison principle, a discrete barrier argument to deal with Dirichlet boundary values, and the property of flatness in one direction within the non-contact set. Our error analysis extends to a modified version of the finite difference wide stencil method of \cite{Ob2}.
\end{abstract}
\maketitle

\vspace{0.2cm}
	

\section{Introduction}
Given an open set $\Omega \subset \mRd$ and a continuous function
$f: \Oc \rightarrow \mR$, its convex envelop in $\Omega$ is defined
as
\begin{equation}\label{E:def-CE}
u(x) = \sup\l\{ l(x): l \leq f \text{ in } \Oc, \; l \text{ is affine} \r\},
\end{equation}
which in fact is the largest convex function majorized by $f$ in $\Oc$.
This function $u$ can also be viewed as the viscosity solution of the following fully nonlinear, degenerate elliptic PDE introduced by Oberman \cite{Ob1}
\begin{equation}\label{E:pde-int-CE}
T[u;f](x) := \min\l\{ f(x) - u(x),  \lam1[D^2u](x)\r\} = 0,
\end{equation}
where $\lam1[D^2u]$ denotes the smallest eigenvalue of the Hessian $D^2u$.
This is the complementarity form of the fully nonlinear obstacle problem at hand.
\begin{figure}[!htb]
\vspace{-10pt}
\includegraphics[width=0.8\textwidth]{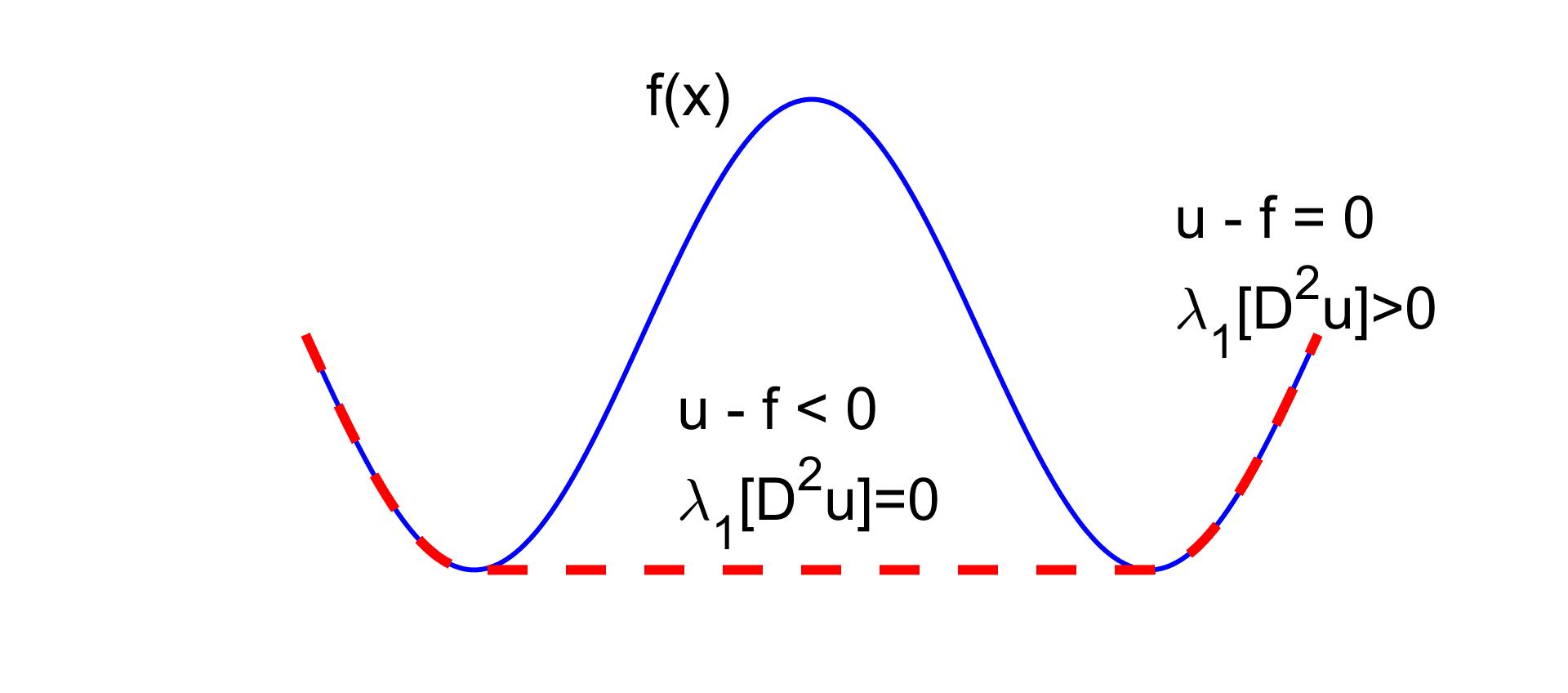}
\vspace{-20pt}
\caption{\small Illustration of the equation \eqref{E:pde-int-CE}. In the non-contact set $\{u < f\}$, the function $u$ must be flat in one direction, i.e. $\lam1[D^2u] = 0$.}
\label{F:pde_CE}
\end{figure}
\Cref{F:pde_CE} illustrates the pde
formulation \eqref{E:pde-int-CE}. Roughly speaking,
in the contact set 
\[
\mC(f) := \l\{ x \in \Oc: u(x) = f(x) \r\},
\]
we have the equality $u = f$ and the inequality $\lam1[D^2u] \geq 0$
given by the convexity of $u$. Outside the contact set, we have
$u < f$ and that $u$ is flat in at least
one direction which implies $\lam1[D^2u] = 0$.

In this paper, we consider the case $\Omega$ bounded and
strictly convex, which guarantees the Dirichlet boundary condition
$u = f$ on $\pO$ is attained. Therefore the convex envelope $u$ of $f$ is the viscosity solution of the following problem:
\begin{equation} \label{E:pde-CE}
\left\{
\begin{aligned}
T[u;f](x) = \min\l\{ f(x) - u(x),  \lam1[D^2u](x)\r\} = 0  & \quad {\rm in} \;\; \Omega,
\\ u =f & \quad {\rm on} \;\; \pO.
\end{aligned}
\right.
\end{equation}

The regularity study of convex envelopes dates back to \cite{TrudUrbas1984, CaNiSp}, thus before the PDE formulation \eqref{E:pde-CE} of \cite{Ob1}. 
However, the problem considered in \cite{TrudUrbas1984, CaNiSp} is a Dirichlet problem for the degenerate \MA equation, $\det(D^2u)=0$, which corresponds to the convex envelope of function $f$ given on the boundary $\pO$ as a Dirichlet condition.
For the convex envelope $u$ in \eqref{E:def-CE}, De Philippis and Figalli \cite{DeFi} obtained recently the
optimal regularity $u \in C^{1,1}(\Oc)$ under the assumption that $\Omega$ is a uniformly convex domain of class $C^{3,1}$ and $f \in C^{3,1}(\Oc)$.

There are a handful of papers regarding the numerical
approximation of convex envelopes. Oberman \cite{Ob2} proposed a
wide stencil method to approximate \eqref{E:pde-int-CE}. Dolzmann \cite{Dolzmann} developed a method to compute rank-one convex envelopes, a related notion of critical importance in materials science. Dolzmann and Walkington \cite{DolzWalk} proved an $O(h^{1/3})$ rate of convergence. Finally, Bartels \cite{Bartels} improved the error estimate of \cite{DolzWalk} to $O(h)$ upon increasing the number of directions and function evaluations within elements,
thus at the expense of extra computational cost.

In this paper, we construct and study
a two-scale method for \eqref{E:pde-CE}, which is
somewhat related to the wide stencil method of \cite{Ob2}. Two-scale methods
are developed in \cite{NoNtZh1}, whereas suboptimal pointwise error estimates are derived in
\cite{NoNtZh2} and optimal ones in \cite{LiNo}. We prove existence, uniqueness, and uniform convergence, as well
as pointwise error estimates under realistic regularity assumptions on $u$. Our proof hinges on a discrete comparison principle and discrete barrier functions, and is thus classical. However, we exploit that $u$ is flat in at least one direction outside the contact set $\mC(f)$ \cite{CaNiSp, ObSi}, a crucial
property that plays an essential role in dealing with low regularity of $u$. Our techniques extend to a modified wide stencil method obtained from that in \cite{Ob2} upon adding a two-scale structure.

The remainder of this paper is organized as follows. In section 2,
we introduce the two-scale method for convex envelope problem \eqref{E:pde-CE} and
prove several properties of it. In section 3, we prove our main
error estimate in the $L^{\infty}$ norm after
reviewing geometric properties of $u$ and studying
the consistency error. We next extend our analysis to a modified wide stencil method
in section 4. We conclude in section 5 with numerical experiments which illustrate the performance
of the two-scale methods and compare with theory.

\section{Two-Scale Method} \label{S:TwoSc}
In this section, we extend the two-scale method developed in
\cite{NoNtZh1} to solve \eqref{E:pde-CE}, and prove several important
properties including convergence.

\subsection{Definition of the Two-Scale Method} \label{S:IntroTwoSc}
%
Let $\{\Th\}$ be a sequence of meshes made of closed simplices $T$. Let $\Th$
be shape-regular and quasi-uniform
with mesh size $h$ and shape-regular constant $\sigma$, i.e.
\begin{equation}\label{E:shape-regularity}
\max_{h} \ \max_{T \in \Th} \ \frac{h_T}{\rho_T} \le \sigma,
\end{equation}
where $h_T$ denotes the diameter of $T$ and
$\rho_T$ the diameter of the largest ball inscribed in $T$.
Let $\Oh$ be the interior of the union of elements $T\in\Th$,
$\Nh$ be the nodes of $\Th$,
$\Nhb := \{x_i \in \Nh: x_i \in \partial \Omega\}$
be the boundary nodes and
$\Nhi := \Nh \setminus \Nhb$ be the interior nodes;
since we require that
$\Nhb \subset \partial \Omega$ we deduce that $\Omega_h\subset\Omega$ is also convex. Let
$\Vh$ be the space of continuous piecewise linear functions over $\Th$.

Before introducing the two-scale method we need additional notation.
Let $\mS$ be the unit sphere in $\mRd$.
We consider a finite discretization $\St \subset \mS$ of $\mS$ governed by the
parameter $\theta$: given any
$v \in \mS$, there exists $v^\theta \in \St$ such that
\begin{equation*}
|v - v^{\theta} | \leq \theta.
\end{equation*}

Let the meshsize $h$ be the fine scale and $\delta \geq h$ (to be
chosen later) be the coarse scale.
For every $x_i\in\Nhi$, let
\begin{equation}\label{E:deltai}
\delta_i := \min\big\{\delta,\dist(x_i,\partial\Omega_h)\big\},
\end{equation}
and observe that $\delta_i\ge C(\sigma) h$ and the open ball
$B(x_i,\delta_i)$ centered at $x_i$ with radius $\delta_i$ is
contained in $\Omega_h$.
For any function $w \in C(\Ohc)$, in particular for $w \in \Vh$,
let the centered second difference operator be
\begin{equation} \label{E:2Sc2Dif}
\sdd w(x_i;v) :=   \frac{ w(x_i+ \delta_i v) -2 w(x_i) + w(x_i- \delta_i v) }{ \delta_i^2}
\end{equation}
and note that it is well defined for all $x_i\in\Nhi$ and $v \in \mS$.
Since
\begin{equation}\label{E:lam1}
\lambda_1[D^2 w](x) = \min_{v \in \mS} \partial_{vv}^2 w(x),
\end{equation}
we consider the following approximation of $\lambda_1[D^2 w]$ at $x=x_i \in \Nhi$
\begin{equation*}
\lambda_1[D^2 w](x_i) \approx \min_{v \in \St} \sdd w(x_i;v).
\end{equation*}
If $\ve := (h,\delta,\theta)$ encodes the discretetization parameters, our two-scale operator $\Tve$
for the convex envelope problem \eqref{E:pde-int-CE} is finally given by
\begin{equation}\label{E:disc-oper}
\Tve[w_h;f](x_i) = \min\l\{f(x_i) - w_h(x_i), \;
\min_{v \in \St} \sdd w_h(x_i;v) \r\}
\quad\forall \,  x_i \in \Nhi
\end{equation}
for any $w_h\in\Vh$. The corresponding two-scale method reads:
seek $\uve \in \Vh$
\begin{equation} \label{E:2ScOp}
\Tve[\uve;f](x_i) = 0 \quad \; \forall x_i \in \Nhi,
\end{equation}
and $\uve(x_i) = f(x_i)$ for all $x_i \in \Nhb$. We say that $w_h \in \Vh$ is
a discrete subsolution (supersolution) of \eqref{E:2ScOp} if
\[
\Tve[w_h; f](x_i) \ge 0 \ (\le 0) \quad \forall x_i \in \Nhi \ ; \quad
w_h(x_i) \le (\ge) f(x_i) \quad \forall x_i \in \Nhb .
\]
Therefore, a discrete solution of \eqref{E:2ScOp} is both a discrete sub and
supersolution.

Although this discrete solution $\uve$ fails to be convex in general,
it is still discretely convex, which is a notion of approximate
convexity introduced in \cite{NoNtZh1}. We say that $w_h \in \Vh$
is \textit{discretely convex} \cite{NoNtZh1} if
\begin{equation*}
\sdd w_h(x_i;v) \geq 0 \qquad \forall x_i \in \Nhi, \quad
\forall v \in \St.
\end{equation*}

\subsection{Discrete Comparison Principle}\label{S:DCP}
One important feature of the definition \eqref{E:disc-oper} of the discrete operator $\Tve$ is its monotonicity. This is similar to the two-scale method for \MA equation in \cite[Lemma 2.3]{NoNtZh1}.

\begin{Lemma}[monotonicity] \label{L:Monotonicity}
Let $x_i \in \Nhi$ be an interior node and $u_h,w_h \in \Vh$. If
$u_h(x_i) \geq w_h(x_i)$ and
\begin{equation*}
\sdd u_h(x_i;v) \leq \sdd w_h(x_i;v)
\end{equation*}
for any $v \in \St$, then
\begin{equation*}
\Tve[u_h;f](x_i) \leq \Tve[w_h;f](x_i).
\end{equation*}
In particular, if $u_h - w_h$
attains a non-negative maximum at $x_i$, then
\begin{equation*}
\Tve[u_h;f](x_i) \leq \Tve[w_h;f](x_i).
\end{equation*}
\end{Lemma}

\begin{proof}
If $\sdd u_h(x_i;v) \leq \sdd w_h(x_i;v)$ for any $v \in \St$, then
\begin{equation*}
\min_{v \in \St} \sdd u_h(x_i;v) \le \min_{v \in \St} \sdd w_h(x_i;v).
\end{equation*}
Recalling the definition \eqref{E:disc-oper} of $\Tve$ and combining
with the fact $u_h(x_i) \geq w_h(x_i)$, this implies
\begin{equation*}
\Tve[u_h;f](x_i) \le \Tve[w_h;f](x_i).
\end{equation*}

On the other hand, if $u_h - w_h$
attains a non-negative maximum at $x_i \in \Nhi$, then we
have $u_h(x_i) \geq w_h(x_i)$ and
\begin{equation*}
u_h(x_i) - w_h(x_i) \geq u_h(z) - w_h(z) \quad \forall z \in \Ohc.
\end{equation*}
By definition \eqref{E:2Sc2Dif} of operator $\sdd$, we obtain
 \begin{equation*}
\sdd u_h(x_i;v) \leq \sdd w_h(x_i;v) \quad \forall v \in \St,
\end{equation*}
and thus use the previous result to conclude the proof.
\end{proof}

Monotonicity leads to the following discrete comparison principle.

\begin{Lemma}[discrete comparison principle] \label{L:DCP}
	Let $u_h,w_h \in \Vh$ with $u_h(x_i) \leq w_h(x_i)$
    for all $x_i \in \Nhb$ and
	\begin{equation}\label{E:comparison}
	\Tve[u_h;f](x_i) \geq \Tve[w_h;f](x_i) \quad \forall x_i \in \Nhi.
	\end{equation}
	Then, $u_h \leq w_h$ in $\Omega_h$.
\end{Lemma}

\begin{proof}
The proof splits into two steps.

\textbf{Step 1.}  We first consider the case with strict inequality
\begin{equation}\label{E:strict}
\Tve[u_h;f](x_i) > \Tve[w_h;f](x_i) \quad \forall x_i \in \Nhi.
\end{equation}
We assume by contradiction that there exists an interior node $x_k \in \Nhi$ such that
$u_h - w_h$ attains a maximum at $x_k$, and $u_h(x_k) > w_h(x_k)$.
Then, by \Cref{L:Monotonicity} (monotonicity) we obtain the contradiction
\begin{equation} \label{E:Contr}
\Tve[u_h;f](x_k) \leq \Tve[w_h;f](x_k).
\end{equation}

\textbf{Step 2.}
Now we deal with \eqref{E:comparison} without the strict inequality. We introduce the auxiliary
strictly convex function $q(x) = \frac{1}{2}|x-x_0|^2- \frac{1}{2}R^2$, which satisfies
$q \leq 0$ on $\overline{\Omega}$, and in particular
$q \leq 0$ on $\partial\Omega_h$ provided
$R = \textrm{diam} (\Omega)$ and $x_0 \in \Omega$.
Its Lagrange interpolant $q_h=\interp q$ is discretely convex and satisfies
\begin{equation*}
\sdd q_h(x_i;v) \geq \sdd q(x_i;v) = \partial^2_{vv} q(x_i) = 1
\quad\forall x_i \in \Nhi,
\quad\forall v \in \St,
\end{equation*}
because $q$ is quadratic. For arbitrary $\alpha>0$,
consider the function
$u_{\alpha} = u_h + \alpha q_h - \alpha$, which satisfies
$u_{\alpha} < u_h \leq w_h$ on $\partial\Oh$ and
\begin{equation*}
\begin{aligned}
\Tve[u_{\alpha};f](x_i)
& = \min\l\{f(x_i) - u_{\alpha}(x_i), \;
\min_{v \in \St} \sdd u_{\alpha}(x_i)(x_i;v) \r\} \\
& \geq \min\l\{f(x_i) - (u_h(x_i) - \alpha), \;
\min_{v \in \St} \l(\sdd u_h(x_i;v) + \alpha \r) \r\} \\
&= \Tve[u_h;f](x_i) + \alpha > \Tve[w_h;f](x_i) \quad \forall x_i \in \Nhi.
\end{aligned}
\end{equation*}
Applying Step 1 we deduce
\begin{equation*}
u_h + \alpha q_h - \alpha \leq w_h \quad \forall \alpha>0.
\end{equation*}
Finally, let $\alpha\to0$ to obtain the asserted inequality.
\end{proof}

\subsection{Existence, Uniqueness and Stability}\label{S:Exist-Uniq}
We now prove several properties of our discrete system \eqref{E:2ScOp}
which are useful for the proof of convergence.

\begin{Lemma} [existence, uniqueness and stability]\label{L:Exist-Uniq-Stab}
There exists a unique $\uve\in\Vh$ that solves the discrete
equation \eqref{E:2ScOp}. The solution $\uve$ is stable in
the sense that $\|\uve\|_{L^{\infty}(\Oh)} \leq \|f\|_{L^{\infty}(\Omega)}$
regardless of the parameters $\ve=(h,\delta,\theta)$ of the method.
\end{Lemma}

\begin{proof}
Since uniqueness is a trivial consequence of Lemma \ref{L:DCP}
(discrete comparison principle), we
just have to prove existence and stability.

\textbf{Step 1 - Stability:}
We first show that $u_h^- = \interp u$ is a discrete subsolution
where $u$ is the exact convex envelope and $u_h^+ = \interp f$
is a discrete supersolution, where again $\interp$ stands for the Lagrange interpolation operator.

Since $u$ is the exact convex envelope,
for any $x_i \in \Nhi$, we have $u_h^-(x_i) \leq f(x_i)$ and
$\sdd u_h^-(x_i;v) \geq 0$ because $u$ is convex. By definition \eqref{E:disc-oper}
of $\Tve$, this gives us $\Tve[u_h^-;f](x_i) \geq 0$ for all $x_i \in \Nhi$.
It is also clear that we have
$\Tve[u_h^+;f](x_i) \leq f(x_i) - u_h^+(x_i) = 0$ for all $x_i \in \Nhi$.
Therefore combining with the fact that
$u_h^+(x_i) = u_h^-(x_i) = f(x_i)$ for $x_i \in \Nhb$, we
see that $u_h^-$ and $u_h^+$ are discrete subsolution and supersolution
respectively. By Lemma \ref{L:DCP} (discrete comparison principle),
this implies
\begin{equation}\label{E:disc-sol-bounds}
u_h^- \leq \uve \leq u_h^+,
\end{equation}
and we thus obtain the stability of $\uve$ because both
$\|u_h^- \|_{L^{\infty}(\Oh)}$ and $\|u_h^+ \|_{L^{\infty}(\Oh)}$
are bounded by $\|f\|_{L^{\infty}(\Omega)}$.

\textbf{Step 2 - Discrete Perron Method:}
It remains to prove the existence of $\uve$. We proceed as in
\cite{NoNtZh1, NoZh2} and use the discrete Perron's method to construct a monotone increasing sequence of functions
$\l\{  u_h^k \r\}_{k=0}^\infty$.
The initial iterate $u_h^0$ is chosen to be $u_h^-$, and
thus satisfies the boundary condition $u_h^0(x_i) = f(x_i)$ for
all $x_i \in \Nhb$ and
\begin{equation}\label{E:discrete-sub}
\Tve[u_h^0;f](x_i) \geq 0 \quad\forall x_i\in\Nhi.
\end{equation}
We construct $\l\{  u_h^k \r\}$
by induction. Suppose that we have already built
$u_h^k\in\Vh$ satisfying both the boundary condition and
\eqref{E:discrete-sub}.
To construct $u_h^{k+1}\in\Vh$ such that $u_h^{k+1}\geq u_h^k$
and also satisfies both the boundary condition and \eqref{E:discrete-sub},
we consider all interior nodes in order and construct auxiliary
functions $u_h^{k,i-1}\in\Vh$ using the first $i-1$ nodes and starting
from $u_h^{k,0}:=u_h^k$ as follows. At $x_i\in\Nhi$ we check whether or not
$\Tve[u_h^{k,i-1};f](x_i) > 0$. If so, we
increase the value of $u_h^{k,i-1}(x_i)$ and denote the resulting
function by $u_h^{k,i}$, until
\begin{equation*}
\Tve[u_h^{k,i};f](x_i) = 0.
\end{equation*}
This is possible because $\Tve[u_h^{k,i};f](x_i)$
is strictly decreasing with respect to $u_h^{k,i}(x_i)$.
Expression \eqref{E:disc-oper} also shows that
this process does not decrease $\Tve[u_h^{k,i};f](x_j)$
for any $x_j \ne x_i$, whence
\begin{equation*}
\Tve[u_h^{k,i};f](x_j) \geq \Tve[u_h^{k,i-1};f](x_j) \geq 0
\quad\forall x_j\ne x_i.
\end{equation*}
We repeat this process with the remaining nodes $x_j$ for $i<j\leq N$
where $N$ is the number of all interior points,
and set $u_h^{k+1} := u_h^{k,N}$ to be the last intermediate
function. By construction, we clearly obtain
\begin{equation*}
\Tve[u_h^{k+1};f](x_i) \geq 0,
\quad
u_h^{k+1}(x_i) \geq u_h^k(x_i)
\quad\forall x_i\in\Nhi,
\end{equation*}
and $u_h^k(x_i) = f(x_i)$ for all $x_i \in \Nhb$.

\textbf{Step 3 - Convergence of $u_h^k$:}
By construction we have $u_h^k \geq u_h^0 = u_h^-$ and by \Cref{L:DCP} (discrete comparison principle),
$u_h^k \leq u_h^+$ and thus $u_h^k(x_i)$ is uniformly bounded.
Since the sequence $\{u_h^k\}_k$ is monotone, it must converge
to a limit
\begin{equation*}
\uve(x_i) = \lim_{k\to\infty} u_h^k(x_i) = \lim_{k\to\infty} u_h^{k,i}(x_i)
\quad\forall x_i\in\Nh.
\end{equation*}
Due to continuity of $\Tve[w_h;f]$ with respect to $w_h(x_j)$, we have
$\Tve[\uve;f](x_i) = \lim_{k\to\infty}\Tve[u_h^{k,i};f](x_i) = 0 $
for any $x_i \in \Nhi$.
This implies that the limit $\uve$ is the solution of discrete equation \eqref{E:2ScOp}
and finishes the proof.
\end{proof}


%
%

Another way to prove existence and uniqueness is to
take advantage of the existing results for Bellman equation and
Howard's algorithm as we can see in section \ref{S:Exp}. \looseness=-1

We define for $x \in \Oc$
\begin{equation}\label{E:def-uo-uu}
\overline{u}(x) := \limsup_{\ve,\frac{h}{\delta} \to 0, \;y \to x} \uve(y),
\quad
\underline{u}(x) := \liminf_{\ve,\frac{h}{\delta} \to 0, \;y \to x} \uve(y),
\end{equation}
where the limits are taken for $y \in \Oh$.
From equation \eqref{E:disc-sol-bounds} and the continuity of
both $u$ and $f$, we immediately obtain the following lemma
characterizing the behavior of $\overline{u}$ and $\underline{u}$
on the boundary $\partial \Omega$.

\begin{Lemma}[boundary behavior] \label{L:disc-sol-boundary}
Let $\Omega$ be a strictly convex bounded domain, let $\uve$
be the discrete solution of \eqref{E:2ScOp}, and let
$\overline{u}(x)$ and $\underline{u}(x)$ be
defined in \eqref{E:def-uo-uu}. Then we have
$\overline{u}(x) = \underline{u}(x) = f(x)$
for all $x \in \partial \Omega$.
\end{Lemma}

\begin{proof}
Since $\Omega$ is strictly convex, the Dirichlet boundary condition $u=f$ on $\pO$ is attained as a direct consequence of \cite[Corollary 17.1.5]{Rockafellar2015convex}, or can be proved in the same way as \cite[Theorem 1.5.2]{Gutierrez}. Next use
\begin{equation*}
    \interp u(x) = u_h^-(x) \leq \uve(x) \leq u_h^+(x) = \interp f(x) \quad x \in \Oh
\end{equation*}
with equality on $\pO$ to deduce the assertion.
\end{proof}

\subsection{Consistency} \label{S:Consistency}
We now quantify the consistency error of our discrete
operator $\Tve[\interp u;f]$ for a smooth function $u$, which is enough for the proof of convergence.
In \Cref{S:RoC} we will carry out
a more delicate analysis of the consistency error which
enables us to prove error estimates for solutions with
weaker but realistic regularity. In the meantime, we stress
that the convex envelope $u$ is 
generically never better than of class $C^{1,1}(\Oc)$ \cite{DeFi}.

Given a node $x_i \in \Nhi$ we denote
\begin{equation}\label{E:Bi}
B_i := \cup \{T: T\in\Th, \, \dist(x_i,T) \leq \delta_i \},
\end{equation}
where $\delta_i$ is defined in \eqref{E:deltai}.
We also denote by $\Ohs$ the following $s$-interior region of $\Oh$ for any parameter $s > 0$
\begin{equation*}
 \Ohs = \left\{ x \in \Oh \; : \;
 \dist(x,\partial \Omega_h ) \geq s \right\}.
\end{equation*}
Hereafter, we use the symbols $C(d,\sigma)$, $C(d)$ and $C$ to
denote constants that depend only on
the dimension $d$ and the shape-regularity constant $\sigma$, but
are independent of the two scales
$h$ and $\delta$, the parameter $\theta$ and the function $u$.

\Cref{L:Consistency-smooth} below establishes a consistency error estimate for the
two-scale method similar to \cite[Lemma 4.1]{NoNtZh1} and \cite[Lemma 4.2]{NoNtZh1}. The 
proof follows along the lines of \cite{NoNtZh1}.

\begin{Lemma} [consistency for smooth functions] \label{L:Consistency-smooth}
Let $u\in \Cka(B_i)$ for $k=0,1$ and $\alpha \in (0,1]$,
$\interp u$ be its Lagrange interpolant,
and $B_i$ be defined in \eqref{E:Bi}.
The following estimates are then valid:
\begin{enumerate}[(i)]
\item
For all $x_i  \in \Nhi$ and all  $v \in \mS$, we have
\begin{equation}\label{E:sddIhuB}
\l|\sdd \interp u (x_i;v) \r| \leq C(d,\sigma) \; |u|_{\Wti(B_i) },
\end{equation}
\item
For all $x_i \in \Nhi \cap \Ohd$ and all $v \in \mS$,
we have
\begin{equation}\label{E:sdd-error}
\l|\sdd \interp u(x_i;v) - \sd{u}{v}(x_i) \r| \leq C(d,\sigma)
\l( |u|_{\Cka(B_i)} \delta^{k+\alpha}+ |u|_{\Wti(B_i)}\frac{h^2}{\delta^2} \r),
\end{equation}
\item
For all $x_i \in \Nhi \cap \Ohd$ and all $v \in \mS$, we have
\begin{equation}\label{E:Op-error-smooth}
\small
\bigg| \Tve[\interp u;f](x_i) - T[u;f](x_i) \bigg| \leq
C(d,\sigma) \left[ |u|_{\Cka(B_i)} \delta^{k+\alpha} +
   |u|_{\Wti(B_i)} \l( \frac{h^2}{\delta^2} + \theta^2 \r) \right].
\end{equation}
\end{enumerate}
\end{Lemma}

\begin{proof}
For the proof of \eqref{E:sddIhuB} and \eqref{E:sdd-error},
the readers may refer to \cite[Lemma 4.1]{NoNtZh1}.
Here we only prove \eqref{E:Op-error-smooth}.

Recalling the definitions of $T$ in \eqref{E:pde-int-CE}
and $\Tve$ in \eqref{E:disc-oper} we only need to prove
{\small\begin{equation*}
\l| \lam1[D^2u](x_i) - \min_{v \in \St} \sdd \interp u(x_i;v) \r|
\leq C(d,\sigma) \left[ |u|_{\Cka(B_i)} \delta^{k+\alpha} +
   |u|_{\Wti(B_i)} \l( \frac{h^2}{\delta^2} + \theta^2 \r) \right].
\end{equation*}}
To this end, first let $\vt$ be the direction such that
\begin{equation*}
\sdd \interp u (x_i;\vt) = \min_{v \in \St} \sdd \interp u (x_i;v).
\end{equation*}
We use \eqref{E:lam1} and \eqref{E:sdd-error} to get
\begin{equation*}
\begin{aligned}
\lam1[D^2u](x_i) - \min_{v \in \St} \sdd \interp u(x_i;v)
\leq & \; \sd{u}{\vt}(x_i) - \sdd \interp u(x_i;\vt) \\
\leq & \; C(d,\sigma)
\l( |u|_{\Cka(B_i)} \delta^{k+\alpha}+ |u|_{\Wti(B_i)} \frac{h^2}{\delta^2} \r),
\end{aligned}
\end{equation*}
which proves one inequality of \eqref{E:Op-error-smooth}.
To show the reverse inequality we let $v$ be the direction that realizes the minimum in \eqref{E:lam1},
which means
\begin{equation*}
\partial_{vv}^2 u(x_i) = \lambda_1[D^2 u](x_i),
\end{equation*}
and we also know that $v$ is the eigenvector of $D^2 u(x_i)$ corresponding to
the smallest eigenvalue $\lam1$.
By definition of $\St$, there exists $\vt \in \St$ such that $|v - \vt| \leq \theta$,
and we can thus write
\begin{equation*}
\min_{v \in \St} \sdd \interp u(x_i;v) - \lam1[D^2u](x_i)
\leq \sdd \interp u(x_i;\vt) - \partial_{vv}^2 u(x_i)
= I_1 + I_2,
\end{equation*}
where
\begin{equation*}
I_1 = \sdd \interp u(x_i;\vt) - \partial_{\vt\vt}^2 u(x_i),
\qquad
I_2 = \partial_{\vt\vt}^2 u(x_i) - \partial_{vv}^2 u(x_i).
\end{equation*}
It is clear that $I_1$ can be bounded by \eqref{E:sdd-error}.
For $I_2$, write $\vt = v + w$, then
\begin{equation*}
\begin{aligned}
\partial_{\vt\vt}^2 u(x_i) &= \vt^T D^2u(x_i) \vt
= \partial_{vv}^2 u(x_i) + 2 w^T D^2u(x_i) v + w^T D^2u(x_i) w \\
&= \partial_{vv}^2 u(x_i) + 2 \lam1 v \cdot w + w^T D^2u(x_i) w.
\end{aligned}
\end{equation*}
Since
\begin{equation*}
1= |\vt |^2 = |v|^2 + 2 v \cdot w + |w|^2,
\end{equation*}
and $|v| = 1$, we observe that
\begin{equation*}
| v \cdot w | = \frac{1}{2} |w|^2 \leq \frac{1}{2}\theta^2,
\end{equation*}
whence we obtain
\begin{equation*}
I_2 \leq C |u|_{\Wti(B_i)} \theta^2.
\end{equation*}
Combining the bounds for both $I_1$ and $I_2$ we have
{\small
\begin{equation*}
\min_{v \in \St} \sdd \interp u(x_i;v) - \lam1[D^2u](x_i)
\leq C(d,\sigma) \left[ |u|_{\Cka(B_i)} \delta^{k+\alpha} +
|u|_{\Wti(B_i)} \l( \frac{h^2}{\delta^2} + \theta^2 \r) \right] .
\end{equation*}}
This finishes the proof of \eqref{E:Op-error-smooth}.
\end{proof}

\subsection{Convergence}\label{S:Convergence}
We are now ready to prove the convergence result.

\begin{Theorem}[convergence] \label{T:Convergence}
If $\Omega$ is a bounded and strictly convex domain and $f \in C(\Oc)$,
then the discrete solution $\uve$ of \eqref{E:2ScOp} converges uniformly
to the convex envelope $u$ of $f$ as $\ve = (h, \delta, \theta) \to 0$ and
$\frac{h}{\delta} \to 0$.
\end{Theorem}
\begin{proof}
Our approximation scheme \eqref{E:2ScOp} satisfies monotonicity (\Cref{L:DCP}),
stability (\Cref{L:Exist-Uniq-Stab}), and consistency
(\Cref{L:Consistency-smooth}). Moreover, the PDE \eqref{E:pde-CE} for the convex envelope problem
admits a comparison principle \cite[Proposition 2.7]{ObRu} for
Dirichlet boundary conditions in the classical sense.
Similarly to \cite[Section 4]{JeSm}, \cite[Theorem 17]{FeJe} and \cite[Section 5]{NoNtZh1},
in order to use the convergence theorem of Barles and Souganidis \cite{BaSoug},
we still need the additional fact that $\overline{u}(x) = \underline{u}(x) = f(x)$
on $\pO$. Since this is proved in \Cref{L:disc-sol-boundary} (boundary behavior), \cite{BaSoug} yields uniform convergence
of the discrete solution $\uve$ to the viscosity solution $u$ of \eqref{E:pde-CE}.
\end{proof}

\section{Rates of Convergence} \label {S:RoC}
In this section, we prove convergence rates for solutions
of class $\nCkao$ for $k=0,1$ and $0<\alpha \leq 1$.
Since in general we could only expect
$u \in C^{1,1}(\Oc)$ even for smooth $f$ and $\Omega$, our estimate
of consistency error in \Cref{S:Consistency} fails.
The challenge is thus to estimate the consistency error
for solutions with less regularity.
We first show a key geometric lemma about convex envelopes which enables us to give an
estimate of the consistency error for $u \in \nCkao$. On the basis on this result, we
next prove the convergence rate using \Cref{L:DCP} (discrete comparison principle).

\subsection{Flatness}\label{S:Flatness}
The heuristic behind the governing PDE \eqref{E:pde-int-CE} is that the convex envelope $u$ must be flat at least in one direction within the non-contact set, i.e. $\lam1[D^2 u](x) = 0$ for all $x \notin  \mC(f)$.
The question whether there is a line segment containing $x$, on which $u$ is flat, is studied in \cite[Section 3]{ObSi} for the Dirichlet convex envelope problem in which $f$ is only defined on $\pO$. For $f \in C(\Oc)$ defined in the entire $\Omega$, and corresponding definition \eqref{E:def-CE} of convex envelope $u$, we have a similar property.

\begin{Lemma}[flatness in one direction]\label{L:Flatness}
Let $f \in C(\Oc)$ and $x \in \Omega$ be such that $\dist(x,\mC(f)) \geq d \delta$.
Then for any slope $p \in \partial u(x)$,
there exists a direction $v \in \mS$ such that
\begin{equation*}
x_{\pm} = x \pm \delta v,
\quad
u(x_{\pm}) = u(x) \pm \delta (p \cdot v),
\quad
\sdd u(x;v) = 0.
\end{equation*}
Moreover, $p$ belongs also to the subdifferential sets
$\partial u(x_{\pm})$.
\end{Lemma}

This lemma says that if $x$ is away from the contact set $\mC(f)$
at least at distance $d \delta$, then there exists a line segment centered at $x_i$ with length at least $2\delta$ such that the convex envelope $u$ is flat on this segment. The flattness means
the second difference of $u$ in this direction is $0$, which plays
an important role in obtaining consistency error for $x$
far away from $\mC(f)$. To prove \Cref{L:Flatness},
we need the following definition and subsequent result:
given $x \in \Omega \setminus \mC(f)$ and $p \in \partial u(x)$, let
\begin{equation*}
    \mC(f;x,p) := \l\{ y \in \Oc: f(y) = u(x) + p \cdot (y-x) \r\},
\end{equation*}
and note that $\mC(f;x,p) \subset \mC(f)$ because $u$ is convex and
$u(y) \ge u(x) + p \cdot (y-x)$ whence $u(y)=f(y)$. The following auxiliary result
is exactly the same as \cite[Lemma 3.3]{DeFi}
and similar to \cite[Lemma 2]{CaNiSp} and \cite[Theorem 3.2]{ObSi}.
We still give a proof here for completeness.

\begin{Lemma}[structure of non-contact set]\label{L:noncontact-set}
Let $f \in C(\Oc)$ and $x \in \Omega \setminus \mC(f)$.
Then for any slope $p \in \partial u(x)$, there exist points
$x_1,\ldots,x_k \in \mC(f)$ with $2 \leq k \leq d+1$ such that
\begin{equation*}
x \in \conv \;(x_1,\ldots,x_k),
\end{equation*}
and $u$ is affine in the convex hull $\conv \;(x_1,\ldots,x_k)$ of $(x_i)_{i=1}^k$. Moreover, $p$ is also in the subdifferential set
$\partial u(y)$ for any $y \in \conv \;(x_1,\ldots,x_k)$.
\end{Lemma}

\begin{proof}
For any $p \in \partial u(x)$, define
$P(y) := u(x) + p \cdot (y-x)$ and observe that
\begin{equation*}
\mC := \mC(f;x,p) = \l\{ y \in \Oc: f(y) = P(y) \r\}.
\end{equation*}
We claim that $x \in \conv(\mC)$. Argue by contradiction,
suppose $x \notin \conv(\mC)$, and use the hyperplane separation theorem
to find an affine function $L$ such that
$L(x) > 0$ and $L(y) < 0$ for every $y \in \mC$. By the definition
of $\mC$ and the fact that $P \leq u \leq f$, it is clear that
$f - P$ is strictly positive in the compact set
$\Oc \cap \{ L \geq 0\}$: in fact, if $f(y) \le P(y)$ then $f(y) = P(y) = u(y)$
and $y \in \mC$, whence $L[y] < 0$. Therefore it is easy to see that
for some small $\alpha > 0$, we have
\begin{equation*}
\wt{L}(y) := P(y) + \alpha L(y) \leq f(y) \quad \forall y \in \Oc,
\end{equation*}
but $\wt{L}(x) > P(x) = u(x)$. This contradicts the definition
of convex envelope $u$ and thus proves the claim $x \in \conv(\mC)$.
Now we use Carath\'{e}odory's theorem to obtain the existence of
$x_1,\ldots,x_k \in \mC$ with $k \leq d+1$ such that
$x \in \conv(x_1,\ldots,x_k)$.

To prove that $p \in \partial u(y)$ for any $y \in \conv(x_1,\ldots,x_k)$,
we define
\begin{equation*}
\mK := \l\{ y \in \Oc:
u(y) = u(x) + p \cdot (y - x)
\r\} = \l\{ y \in \Oc: u(y) = P(y)
\r\},
\end{equation*}
whence $u$ is affine in $\mK$. We claim that $\mK$ is convex.
Let $y_1,y_2 \in \mK, \lambda \in (0,1)$
and $z = \lambda y_1 + (1- \lambda) y_2$. Since $u$ is convex,
we have
\begin{equation*}
u(z) \leq \lambda u(y_1) + (1- \lambda) u(y_2)
= \lambda P(y_1) + (1- \lambda) P(y_2) = P(z).
\end{equation*}
On the other hand, since $p \in \partial u(x)$,
the supporting plane $P$ must be below $u$,
and in particular
\begin{equation*}
u(z) \geq P(z).
\end{equation*}
Therefore $u(z) = P(z)$, and thus $z \in \mK$, which implies
the convexity of $\mK$. Since $P \leq u \leq f$,
we have $\{x_1, \ldots, x_k \} \subset \mC \subset \mK$
and $\conv \;(x_1, \ldots, x_k) \subset \mK$.
It is clear that for any $y \in \mK$, we have
$u(y) = u(x) + p \cdot (y - x)$ and
\begin{equation*}
P(z) = u(x) + p \cdot (z - x)
= u(y) + p \cdot (z - y) \leq u(z) \quad \forall z \in \Oc.
\end{equation*}
By definition of $\partial u(y)$ this implies $p \in \partial u(y)$
for any $y \in \conv(x_1,\ldots,x_k)$. In addition, $u$ is
affine in $\conv \;(x_1, \ldots, x_k)$.
\end{proof}

\begin{proof}[Proof of \Cref{L:Flatness}]
For any $p \in \partial u(x)$, by \Cref{L:noncontact-set} (structure of non-contact set), there exist $k$ ($2\le k \le d+1$) points $x_i \in \mC(f;x,p)$ such that
\begin{equation*}
x = \sum_{i=1}^{k} \lambda_i x_i, \quad
\lambda_i \geq 0, \quad
\sum_{i=1}^{k} \lambda_i = 1,
\end{equation*}
and $p$ belongs to the subdifferential set $\partial u(y)$
for any $y \in \conv(x_1,\ldots,x_k)$. If $j$ is such that
$\lambda_j = \max_{1 \leq i \leq k} \lambda_i$, then we have
\begin{equation*}
\lambda_j \geq \frac{1}{k} \sum_{i=1}^{k} \lambda_i
= \frac{1}{k} \geq \frac{1}{d+1}.
\end{equation*}
Now let $x_0
= \sum_{i \neq j} \frac{\lambda_i}{1-\lambda_j} x_i \in
\conv \;(x_1, \ldots, x_k)$ to get
\begin{equation*}
x = \sum_{i=1}^{k} \lambda_i x_i =
\lambda_j x_j + \sum_{i \neq j} \lambda_i x_i =
\lambda_j x_j + (1-\lambda_j) x_0.
\end{equation*}
Since both $x_0,x_j \in \conv \;(x_1,\ldots,x_k)$, the segment
$\overline{x_0x_j}$ is also in $\conv \;(x_1,\ldots,x_k)$.
Due to the fact $\dist(x,\mC(f)) \geq d \delta$,
we have $|x_j - x| \geq d \delta$, and
\begin{equation*}
|x_0 - x| = \frac{\lambda_j}{1 - \lambda_j} |x_j - x|
\geq \frac{1/(d+1)}{1-1/(d+1)} \; d \delta = \delta.
\end{equation*}
Therefore, if $v = \frac{x_j - x}{|x_j - x|}$ and
$x_{\pm} = x \pm \delta v$, clearly $x_{\pm}$
lie in the segment $\overline{x_0x_j}$, and thus
also inside $\conv(x_1,\ldots,x_k)$.
Finally, \Cref{L:noncontact-set} (structure of non-contact set) shows $p \in \partial u(x_{\pm})$ and
$u(x_{\pm}) = u(x) \pm \delta (p \cdot v)$,
which immediately leads to $\sdd u(x;v) = 0$.
\end{proof}

\subsection{Consistency for Solutions with H\"older Regularity}\label{S:Consistency-lowreg}
In this section, we take advantage of results in \Cref{S:Flatness} to
derive a consistency error for solutions with realistic H\"older
regularity $u \in \nCkao$ for $k=0,1$
and $0<\alpha \leq 1$, which improves upon the consistency error
estimates in \Cref{S:Consistency}.

The Lagrange interpolant $\interp u \in \Vh$ of $u$ satisfies for all interior nodes $x_i \in \Nhi$
\begin{equation*}
\interp u(x_i) = u(x_i) \leq f(x_i), \quad
\sdd \interp u(x_i;v) \geq \sdd u(x_i;v) \geq 0 \quad \forall v \in \mS
\end{equation*}
because of the convexity of $u$. In view of definition \eqref{E:disc-oper}
of $\Tve$, this in turn implies $\Tve[\interp u;f](x_i) \geq 0$ for all $x_i \in \Nhi$.
The following proposition yields upper bounds for $\Tve[\interp u;f](x_i)$ depending on the location of
$x_i$ relative to $\mC(f)$ and $\pO$.
%
%
\begin{Proposition}[consistency for $u$ with H\"older regularity]\label{Prop:Consistency-lowreg}
Let $\Omega$ be a bounded strictly convex domain,
$u \in \nCkao$ for $k=0,1$ and $0<\alpha \leq 1$
be the exact solution of the convex envelope problem \eqref{E:pde-CE}.
In addition, let $B_i$ be defined in \eqref{E:Bi} and set
\begin{equation}\label{E:wtBi}
\wt{B_i} := \{x \in \Oc: |x - x_i| \leq d\delta \}.
\end{equation}
For $x_i \in \Nhi$,
the following estimates are then valid:

\begin{enumerate}[(i)]
\item If $\dist(x_i,\mC(f)) \geq d \delta$, we have
\begin{equation}\label{E:error-noncontact}
\min_{\vt \in \St} \sdd \interp u(x_i;\vt) \leq C(d,\sigma)
\frac{(\delta \theta)^{k+\alpha} + h^{k+\alpha}}{\delta^2}
|u|_{\nCka(B_i)}.
\end{equation}
\item If $\dist(x_i,\mC(f)) < d \delta, \; \dist(x_i,\pO) \geq d\delta$,
and $f \in \nCkao$, then for $k = 0$ we have
\begin{equation}\label{E:error-contact-0}
f(x_i) - u(x_i) \leq C(d,\sigma) \delta^{\alpha}
\l( |u|_{\Czeroa(\wt{B_i})} + |f|_{\Czeroa(\wt{B_i})} \r),
\end{equation}
whereas for $k = 1$ we have
\begin{equation}\label{E:error-contact-1}
f(x_i) - u(x_i) \leq C(d,\sigma) \delta^{1+\alpha}
|f|_{\Conea(\wt{B_i})}.
\end{equation}
\item If $0 < \dist(x_i,\pO) < d\delta$, then for all $v \in \mS$,
we have
\begin{equation}\label{E:error-boundary}
\sdd \interp u(x_i;v) \leq C(d,\sigma) \delta_i^{k+\alpha-2} |u|_{\nCka(B_i)},
\end{equation}
and \eqref{E:error-contact-0} also holds provided $k=0$.
\end{enumerate}
\end{Proposition}

\begin{proof}
Since $\Omega$ is strictly convex, we have $\pO \subset \mC(f)$. This implies
that $x_i \in \Nhi$ must fall within one of the following three mutually exclusive cases.


\textbf{Case 1: } $\dist(x_i,\mC(f)) \geq d \delta$. By \Cref{L:Flatness} (flatness in one direction),
for any $p \in \partial u(x_i)$, there exists
$v \in \mS$ such that
\begin{equation*}
x_{\pm} = x_i \pm \delta v,
\quad
u(x_{\pm}) = u(x_i) \pm \delta (p \cdot v),
\quad
\sdd u(x_i;v) = 0.
\end{equation*}
By the definition of $\St$, there exists $\vt \in \St$ such that $|v - \vt| \leq \theta$.
We claim that
\begin{equation*}\sdd \interp u(x_i;\vt) \leq C(d,\sigma)
\frac{(\delta \theta)^{k+\alpha} + h^{k+\alpha}}{\delta^2}
|u|_{\nCka(B_i)},
\end{equation*}
which implies \eqref{E:error-noncontact}. Using
$\; \dist(x_i,\mC(f)) \geq d \delta$, we have $\delta_i = \delta$
in definition \eqref{E:2Sc2Dif}.
Let $x^{\theta}_{\pm} = x_i \pm \delta \vt$, then
$x^{\theta}_{\pm} \in B_i$ and $|x^{\theta}_{\pm} - x_{\pm} | \leq \delta \theta$.
Since the interpolation error satisfies
\begin{equation}\label{E:interp-error}
|u - \interp u|_{L^{\infty}(B_i)} \leq C(d,\sigma) h^{k+\alpha}|u|_{\nCka(B_i)},
\end{equation}
we infer that
\begin{equation}\label{E:sdd-u-uh}
\l| \sdd \interp u(x_i;\vt) - \sdd u(x_i;\vt) \r| \leq
C(d,\sigma) \frac{h^{k+\alpha}}{\delta^2} |u|_{\nCka(B_i)},
\end{equation}
whence it remains to prove
\begin{equation*}
\sdd u(x_i;\vt) \leq C(d,\sigma) \frac{(\delta \theta)^{k+\alpha}}{\delta^2}
|u|_{\nCka(B_i)}.
\end{equation*}

For $k = 0$,
by definition of $|u|_{\nCka(B_i)}$ seminorm, we see that
\begin{align*}
|u(x_{\pm}) - u(x^{\theta}_{\pm})| \leq |x^{\theta}_{\pm} - x_{\pm} |^{\alpha} \; |u|_{\nCka(B_i)}
\leq (\delta \theta)^{\alpha} \; |u|_{\nCka(B_i)}.
\end{align*}
Using this inequality, along with $\sdd u(x_i;v) = 0$, yields the desired bound
\begin{equation*}
\begin{aligned}
\sdd u(x_i;\vt) & \leq \sdd u(x_i;v) +
\frac{|u(x_+) - u(x^{\theta}_+)| + |u(x_+) - u(x^{\theta}_+)|}{\delta^2} \\
& \leq \frac{2(\delta \theta)^{k+\alpha}}{\delta^2}
|u|_{\nCka(B_i)}.
\end{aligned}
\end{equation*}

For $k=1$, we know $p = \nabla u(x_i) = \nabla u(x_{\pm})$.
If $w = \vt - v$, we then have
\begin{equation*}
\begin{aligned}
u(x^{\theta}_{\pm}) &= u(x_{\pm}) \pm \int_0^1 \delta \;
\nabla u\l(x_{\pm} \pm t\delta w \r) \cdot w \; dt \\
&= u(x_{\pm}) \pm \delta \nabla u(x_{\pm}) \cdot w
\pm \int_0^1 \delta \;
\l[ \nabla u\l(x_{\pm} \pm t\delta w \r) - \nabla u(x_{\pm}) \r]
\cdot w \; dt,
\end{aligned}
\end{equation*}
whence
\begin{equation}\label{E:C1a-error}
\begin{aligned}
u(x^{\theta}_{\pm}) &\leq u(x_{\pm}) \pm \delta \nabla u(x_{\pm}) \cdot w
+ \int_0^1 \delta \;
|t \delta w|^{\alpha}|u|_{\nCka(B_i)} \; |w|  \; dt \\
&\leq u(x_{\pm}) \pm \delta p \cdot w
+ C (\delta \theta)^{1+\alpha} |u|_{\nCka(B_i)} .
\end{aligned}
\end{equation}
Therefore plugging the above inequalities into the expression
of $\sdd u(x_i;\vt)$ we obtain
\begin{equation*}
\begin{aligned}
\sdd u(x_i;\vt) &\leq \sdd u(x_i;v) +
\frac{1}{\delta^2} \l( \delta p \cdot w - \delta p \cdot w
+ 2 C (\delta \theta)^{1+\alpha} |u|_{\nCka(B_i)} \r) \\
&\leq C \frac{(\delta \theta)^{1+\alpha}}{\delta^2} |u|_{\nCka(B_i)},
\end{aligned}
\end{equation*}
and finish the proof of our claim.
%

\textbf{Case 2: } $\dist(x_i,\mC(f)) < d \delta$ and $\; \dist(x_i,\pO) \geq d\delta$.
By the assumptions, there exists $y \in \mC(f) \setminus \pO$ such that
$|x_i - y| < d \delta$. We claim that if $k = 0$,
\begin{equation*}
f(x_i) - \interp u(x_i) \leq C(d,\sigma) \delta^{\alpha}
\l( |u|_{\Czeroa(\wt{B_i})} + |f|_{\Czeroa(\wt{B_i})} \r),
\end{equation*}
which is \eqref{E:error-contact-0}. This claim is a
consequence of $\interp u(x_i) = u(x_i), u(y) = f(y)$ and
\begin{equation*}
\begin{aligned}
&\l| u(x_i) - u(y) \r| \leq |x_i - y|^{\alpha} |u|_{\Czeroa(\wt{B_i})}
\leq d^{\alpha} \delta^{\alpha} |u|_{\Czeroa(\wt{B_i})}, \\
&\l| f(x_i) - f(y) \r| \leq |x_i - y|^{\alpha} |f|_{\Czeroa(\wt{B_i})}
\leq d^{\alpha} \delta^{\alpha} |f|_{\Czeroa(\wt{B_i})}.
\end{aligned}
\end{equation*}

If $k = 1$, we claim that
\begin{equation*}
f(x_i) - \interp u(x_i) \leq C(d,\sigma) \delta^{1+\alpha} |f|_{\Conea(\wt{B_i})},
\end{equation*}
which is \eqref{E:error-contact-1}. To prove this claim, we let $p = \nabla u(y)$,
then consider the supporting hyperplane
$P(x) := u(y) + (x - y) \cdot p$. Since $f$
is differentiable, $f(y) = P(y)$ and
$f(x) \geq u(x) \geq P(x)$, we know $p = \nabla f(y)$.
Proceeding similarly to \eqref{E:C1a-error}, we end up with
\begin{equation*}
\l| f(x_i) - P(x_i) \r| =
\l| f(x_i) - f(y) - (x_i - y) \cdot p \r|
\leq C(d,\sigma) \delta^{1+\alpha}
|f|_{\Conea(\wt{B_i})}.
\end{equation*}
Therefore our claim holds because
\begin{equation*}
f(x_i) - \interp u(x_i) = f(x_i) - u(x_i)
\leq f(x_i) - P(x_i)
\leq C(d,\sigma) \delta^{1+\alpha}
|f|_{\Conea(\wt{B_i})}.
\end{equation*}
%
\textbf{Case 3:} $0 < \dist(x_i,\pO) < d\delta$.
We point out that, unlike the first two cases, the upper bound given in \eqref{E:error-boundary}
does not converge to zero as $\delta_i \to 0$. However, this result is still useful in
our proof of error estimates.
We claim that for all $v \in \mS$,
\begin{equation*}
\sdd \interp u(x_i;v) \leq C(d,\sigma) \delta_i^{k+\alpha-2} |u|_{\nCka(B_i)},
\end{equation*}
which is \eqref{E:error-boundary}. Using \eqref{E:interp-error} and the fact
$ \delta_i/h \geq C(d,\sigma)$ due to the shape-regularity assumption on the mesh $\Th$, we have
\begin{equation*}
\begin{aligned}
\l| \sdd u(x_i;v) - \sdd \interp u(x_i;v) \r|
&\leq C(d,\sigma) \frac{h^{k+\alpha}}{\delta_i^2} |u|_{\nCka(B_i)} \\
&\leq C(d,\sigma) \delta_i^{k+\alpha - 2} |u|_{\nCka(B_i)}.
\end{aligned}
\end{equation*}
Consequently, it just suffices to prove
\begin{equation*}
\sdd u(x_i;v) \leq C(d,\sigma) \delta_i^{k+\alpha-2} |u|_{\nCka(B_i)}.
\end{equation*}
If $k = 0$, this is obtained from
\begin{equation*}
\l| u(x_i \pm \delta_i v) - u(x_i) \r|
\leq \delta_i^{\alpha} |u|_{\Czeroa(B_i)}.
\end{equation*}
If $k = 1$, let $p = \nabla u(x_i)$ and $P(x) = u(x_i) + (x-x_i) \cdot p$, we have similarly to \eqref{E:C1a-error}
\begin{equation*}
\l| (u-P)(x_i \pm \delta_i v) \r|
\leq C \delta_i^{1+\alpha} |u|_{\Conea(B_i)}.
\end{equation*}
Therefore since $\sdd P(x_i;v) = 0$, our claim is a consequence of
\begin{equation*}
\sdd u(x_i;v) \leq
\sdd P(x_i;v) + \frac{C \delta_i^{1+\alpha} |u|_{\Conea(B_i)}}{\delta_i^2}
= C \delta_i^{\alpha-1} |u|_{\Conea(B_i)}.
\end{equation*}
This concludes the proof.
\end{proof}

\subsection{Discrete Barrier Functions}\label{S:DBarrier}
In \Cref{Prop:Consistency-lowreg} (consistency for $u$ with H\"older regularity) we estimate the consistency error
for the convex envelope $u \in \nCkao$ for $k=0,1$
and $0<\alpha \leq 1$. In order to take advantage of this result
for error analysis, we now introduce two discrete barrier
functions. The first one is used to handle those $x_i \in \Nhi$
far from the contact set $\mC(f)$, which satisfy the condition
in \Cref{Prop:Consistency-lowreg}(i). The second discrete barrier
function is used to handle those $x_i \in \Nhi$ close to the boundary of $\Omega$,
which satisfy the condition in \Cref{Prop:Consistency-lowreg}(iii).

First we collect properties of the discrete barrier function $q_h$ introduced
in the proof of \Cref{L:DCP} (discrete comparison principle);
see also \cite[Lemma 4.1]{LiNo}.
\begin{Lemma}[discrete barrier $q_h$]\label{L:Barrier_qh}
Let $x_0 \in \Omega$ and $R = \textrm{diam} (\Omega)$.
The interpolant $q_h = \interp q \in \Vh$ of the function $q(x) = \frac{1}{2}|x-x_0|^2 - \frac{1}{2}R^2$ satisfies
\begin{subequations}
\begin{align}
\label{E:Barrier_qh-1}
\sdd q_h(x_i;v_j) \geq 1 \quad &\forall \; x_i \in \Nhi, \; v_j \in \mS, \\
\label{E:Barrier_qh-2}
-C \leq \; q_h(x) \; \leq 0 \quad &\forall \; x \in \Oh,
\end{align}
\end{subequations}
where constant $C$ only depends on $\Omega$.
\end{Lemma}
 
Now we construct our second discrete barrier function $p_h(x)$.
For $k=0,1$ and $0<\alpha \leq 1$, $p_h$ is to satisfy the property
\begin{equation*}
\max_{\vt \in \St} \sdd p_h(x_i;\vt) \geq \delta_i^{k+\alpha-2},
\quad \forall \; x_i \in \Nhi \setminus \Ohdd.
\end{equation*}
We consider a convex function $\eta: [0,\infty) \rightarrow (-\infty,0]$ satisfying
\begin{equation}\label{E:eta-prop-1}
\eta''(t) = 2^{4-k-\alpha} \ t^{k+\alpha-2} \quad t \in (0,2d\delta);
\quad \eta(0) = 0; \quad
\eta'(t) = 0 \quad t \geq 2d\delta.
\end{equation}
Simple calculations reveal that for $k + \alpha \neq 1$,
\begin{equation*}
\eta(t) = \left\{
\begin{array}{ll}
\frac{2^{4-k-\alpha}}{k+\alpha-1}
\l( \frac{1}{k+\alpha}t^{k+\alpha} -
(2d\delta)^{k+\alpha-1}t \r) \quad & \quad 0 \leq t \leq 2d\delta\\
-\frac{16}{k+\alpha} (d\delta)^{k+\alpha}
\quad & \quad t > 2d\delta,
\end{array}
\right.
\end{equation*}
and for $k + \alpha = 1$,
\begin{equation*}
\eta(t) = \left\{
\begin{array}{ll}
8t \l( \ln{t} - \ln(2d\delta) - 1 \r)
\quad & \quad 0 \leq t \leq 2d\delta \\
-16d\delta \quad & \quad t > 2d\delta.
\end{array}
\right.
\end{equation*}
It can be seen immediately that $\eta$ is monotonically non-increasing, and satisfies
\begin{equation}\label{E:eta-prop-2}
-C \delta^{k+\alpha} \leq \eta(t) \leq 0 \qquad \forall t \geq 0.
\end{equation}
Then we define the barrier function $p_h$ as
\begin{equation}\label{E:def_p}
p(x) := \eta(\dist(x, \partial \Oh)) \quad x \in \Oh,
\end{equation}
and denote by $p_h = \interp p \in \Vh$ its Lagrange interpolant.
The following lemma is similar to \cite[Section 6.2]{NoZh1} and
\cite[Lemma 4.2]{LiNo}.
\begin{Lemma}[discrete barrier $p_h$]\label{L:Barrier_p}
If $\Omega$ is strictly convex
and $\theta \leq 1$, then
the discrete barrier function $p_h$
defined in \eqref{E:def_p} satisfies
\begin{subequations}
\begin{gather}\label{E:Barrier_ph-1}
\max_{\vt \in \St} \sdd p_h(x_i;\vt) \geq \; C\delta_i^{k+\alpha-2}
\quad \forall \; x_i \in \Nhi \setminus \Ohdd, \\
\label{E:Barrier_ph-2}
\sdd p_h(x_i;v) \geq \; 0 \quad \forall \;
x_i \in \Nhi, \; v \in \mS , \\
\label{E:Barrier_ph-3}
-C\delta^{k+\alpha}  \leq \; p_h(x) \; \leq 0 \quad \forall \;
x \in \Oh.
\end{gather}
\end{subequations}
Moreover, for $x_i \in \Nhi \setminus \Ohdd$, we could choose $\vt \in \St$ only depending on $x_i, \St$
to satisfy $\sdd p_h(x_i;\vt) \geq \; \delta_i^{k+\alpha-2}$.
\end{Lemma}

\begin{proof}
We proceed as in \cite[Lemma 4.2]{LiNo}.
We first study the function $p$ defined on the convex domain $\Oh \subset \Omega$; the properties of $p_h$ will be simple consequences of those of $p$. Define $d(x) := \dist(x,\Oh)$ for any $x \in \Oh$.
Given any $x_0 \in \Omega_h$, let $y \in \partial \Oh$ be a
(closest) point so that
\begin{equation*}
|y - x_0| = d(x_0).
\end{equation*}
Since $\Oh$ is convex,
there exists a supporting hyperplane $P$ of $\Oh$ touching $\Oh$ at $y$ and
perpendicular to $\nu := \frac{x_0 - y}{|x_0 - y|}$.
Consider any two points $x_+, x_- \in\Oh$
so that $x_0 = (x_+ + x_-)/2$. Then there exists a
vector $v$ such that $x_{\pm} = x_0 \pm v$ and,
without loss of generality, $\langle v, \nu \rangle \geq 0$; hence
\begin{equation}\label{E:proof-ph}
d(x_{\pm}) \leq \dist(x_{\pm},P) = d(x_0) \pm \langle v, \nu \rangle.
\end{equation}
We now show that $p(x)$ is convex.
We exploit that $\eta$ is a nonincreasing convex function,
and $d(x_0) - \langle v, \nu \rangle \geq 0$, to write
\begin{equation*}
\begin{aligned}
p(x_+) + p(x_-) \geq \;
\eta \l( d(x_0) + \langle v, \nu \rangle \r)
+ \eta \l( d(x_0) - \langle v, \nu \rangle \r)
\geq \; 2 \eta \l(d(x_0) \r)
= \; 2 p(x_0).
\end{aligned}
\end{equation*}
Since this holds for any $x_{\pm}, x_0$ satisfying $x_0 = (x_+ + x_-)/2$,
we deduce that $p(x)$ is convex in $\Oh$. This immediately implies
\eqref{E:Barrier_ph-2}:
\[
\sdd p_h(x_i;v) \geq \; \sdd p(x_i;v) \ge \; 0 \quad \forall \;
x_i \in \Nhi, \; v \in \mS.
\]

We next prove \eqref{E:Barrier_ph-1}.
If $x_i \in \Nhi \setminus \Ohdd$, then
$\delta_i \leq d(x_i) \leq d\delta_i \le d\delta$ and
$d(x_i) \pm \delta_i \in [0,2d(x_i)] \subset [0,2d\delta]$, where $\delta_i \leq \delta$ is defined in \eqref{E:deltai}.
It follows from the definition \eqref{E:def_p} of $p$, inequality \eqref{E:proof-ph} and the monotonicity of $\eta$ that
\begin{equation*}
\begin{aligned}
\sdd p_h(x_i;v) \geq & \sdd p(x_i; v) =
\frac{p(x_i+ \delta_i v)+p(x_i- \delta_i v)-2p(x_i)}{\delta_i^2} \\
\geq & \frac{\eta \l(d(x_i) + \delta_i \langle v,\nu \rangle \r)
+ \eta \l(d(x_i) - \delta_i \langle v,\nu \rangle \r)
-2\eta \l(d(x_i)\r)}{\delta_i^2},
\end{aligned}
\end{equation*}
for all $v \in \mS$. Using the fact that for $t \in [0,2d(x_i)]$,
\begin{equation*}
\eta''(t) \geq 2^{4-k-\alpha} \l(2d(x_i)\r)^{k+\alpha-2}
= 4d(x_i)^{k+\alpha-2} \geq 4 (d\delta_i)^{k+\alpha-2},
\end{equation*}
Taylor expansion gives
\begin{equation*}
\begin{aligned}
\sdd p_h(x_i;v) 
\ge & \frac{\eta''(\xi) \l(\delta_i \langle v,\nu \rangle \r)^2}{\delta_i^2} 
\ge \frac{4 (d\delta_i)^{k+\alpha-2} \; \delta_i^2 \langle v,\nu \rangle^2}{\delta_i^2}
= 4 \langle v,\nu \rangle^2 (d\delta_i)^{k+\alpha-2},
\end{aligned}
\end{equation*}
where $\xi \in (0, 2d(x_i))$.
By definition of $\St$, there exists $\vt \in \St$ such that
$|\vt - \nu| \leq \theta \leq 1$, whence
\begin{equation*}
\langle \vt, \nu \rangle = \frac{|\vt|^2+|\nu|^2 - |\vt - \nu|^2}{2}
\geq \frac{1}{2},
\end{equation*}
which yields $\sdd p_h(x_i;\vt) \geq 4 \langle v_{\theta},\nu \rangle^2 (d\delta_i)^{k+\alpha-2} \geq C\delta_i^{k+\alpha-2} $.
This proves \eqref{E:Barrier_ph-1}, whereas \eqref{E:Barrier_ph-3} is a direct consequence of \eqref{E:eta-prop-2}.
\end{proof}

\begin{remark}[boundary resolution]
Notice that we only assume $\theta \leq 1$ here. Our two-scale method can actually be generalized in such a way that each $x_i \in \Nhi$ has a different choice of $\St(x_i)$.
In fact, in our derivation
of error estimate later, for those $x_i$ with
$\dist(x_i,\pO) < d\delta$, we only require the $\St(x_i)$ to satisfy
requirements of discretization for $\theta \leq 1$. This means in practice, for nodes
near the boundary $\pO$, we do not need as many directions as for the nodes in the interior region.
\end{remark}

\subsection{Error Estimates for Solutions with H\"older Regularity}\label{S:RatesHolder}
In this subsection we deal with solutions
$u$ of \eqref{E:pde-CE} of class $\nCkao$ for $k=0,1$ and $0<\alpha \leq 1$,
and derive convergence rates in the $L^{\infty}$ norm.
Our main analytic tool is \Cref{L:DCP} (discrete comparison principle),
along with the results of Sections \ref{S:Consistency-lowreg}
and \ref{S:DBarrier}.

\begin{Theorem}[error estimate]\label{T:error-estimate}
Let $\Omega$ be strictly convex.
Let $u$ be the viscosity solution of \eqref{E:pde-CE} and $\uve$ be the discrete
solution of \eqref{E:2ScOp}. If $u \in \nCkao$ for $k=0,1$
and $0<\alpha \leq 1$, and $\theta \leq 1$,
there exists a constant $C = C(\Omega,d,\sigma)$ such that
{\small
\begin{equation}\label{E:error-estimate}
\Vert \interp u - \uve \Vert_{L^{\infty}(\Oh)}
\leq C \left[
|u|_{\nCkao} \frac{(\delta \theta)^{k+\alpha} + h^{k+\alpha}+ \delta^{2+k+\alpha}}{\delta^2}
+ |f|_{\nCkao} \delta^{k+\alpha} \right].
\end{equation}}
\end{Theorem}

\begin{proof}
We find lower and upper bounds of
$\uve$ in terms of $\interp u$. For the lower bound,
we recall that $u_h^- = \interp u$ is a discrete subsolution
of \eqref{E:2ScOp} and satisfies $u_h^- \leq \uve$ from
\eqref{E:disc-sol-bounds} in the proof of \Cref{L:Exist-Uniq-Stab}
(existence, uniqueness and stability), thereby yielding a lower bound of $\uve$.

For the upper bound, we construct a discrete supersolution $u_h^+ \in \Vh$ such that
\begin{equation*}
\left\{
\begin{aligned}
\Tve[u_h^+;f](x_i) &\leq 0 \quad \forall x_i \in \Nhi \\
u_h^+(x_i) &\geq f(x_i) \quad \forall x_i \in \Nhb,
\end{aligned}
\right.
\end{equation*}
upon suitably modifying $\interp u$.
We let $u_h^+ \in \Vh$ be of the form
\begin{equation*}
u_h^+ = \interp u - K_1 q_h + K_2 - K_3 p_h,
\end{equation*}
where $q_h, p_h \leq 0$ in $\Oh$ according to \eqref{E:Barrier_qh-2}
and \eqref{E:Barrier_ph-3}, and the positive constants
$K_1, K_2, K_3$ are to be chosen properly. Since
\begin{equation*}
u_h^+(x_i) \geq \interp u(x_i) = f(x_i) \quad\forall \; x_i \in \Nhb,
\end{equation*}
to guarantee that $u_h^+$ is a discrete supersolution,
it remains to show $\Tve[u_h^+;f](x_i) \leq 0$
for all $x_i \in \Nhi$. We divide the subsequent discussion into three cases
based on the position of $x_i$ relative to $\mC(f)$ and $\pO$,
exactly as in \Cref{Prop:Consistency-lowreg}.

If $\dist(x_i,\mC(f)) \geq d \delta$, using the estimate
\eqref{E:error-noncontact} of \Cref{Prop:Consistency-lowreg} (consistency for $u$ with H\"older regularity) and the properties \eqref{E:Barrier_qh-1} of $q_h$ and \eqref{E:Barrier_ph-2} of $p_h$, we have
\begin{equation*}
\begin{aligned}
\min_{v \in \St} \sdd u_h^+(x_i;v)
&\leq \min_{\vt \in \St} \sdd [\interp u - K_1 q_h](x_i;v)
\leq \min_{\vt \in \St} \sdd \interp u(x_i;v) - K_1 \\
&\leq C(d,\sigma) \frac{(\delta \theta)^{k+\alpha} + h^{k+\alpha}}{\delta^2}
|u|_{\nCka(B_i)} - K_1 \leq 0,
\end{aligned}
\end{equation*}
provided that $K_1 = C(d,\sigma) \frac{(\delta \theta)^{k+\alpha}
+ h^{k+\alpha}}{\delta^2} |u|_{\nCkao}$.
Consequently,
\begin{equation*}
\Tve[u_h^+;f](x_i) \leq \min_{v \in \St} \sdd u_h^+(x_i;v) \leq 0.
\end{equation*}

If $\; \dist(x_i,\mC(f)) < d \delta, \; \dist(x_i,\pO) \geq d\delta$,
from \eqref{E:error-contact-0} and \eqref{E:error-contact-1}
in \Cref{Prop:Consistency-lowreg}, we have
\begin{equation*}
\begin{aligned}
f(x_i) - u_h^+(x_i) &\leq f(x_i) - \interp u(x_i) - K_2 \\
&\leq C(d,\sigma) \delta^{k+\alpha}
\l( |u|_{\nCka(\wt{B_i})} + |f|_{\nCka(\wt{B_i})} \r) - K_2 \leq 0,
\end{aligned}
\end{equation*}
with $K_2 = C(d,\sigma) \delta^{k+\alpha} \l( |u|_{\nCka(\overline{\Omega})}
+ |f|_{\nCka(\overline{\Omega})} \r)$. This implies
$\Tve[u_h^+;f](x_i) \leq f(x_i) - u_h^+(x_i) \leq 0$.

If $\; \dist(x_i,\pO) < d\delta$, we have $x_i \in \Nhi \setminus \Omega_{h,d\delta}$. Choosing  $K_3 = C(d,\sigma)|u|_{\nCkao}$ and invoking \eqref{E:error-boundary} in \Cref{Prop:Consistency-lowreg} and
the property \eqref{E:Barrier_ph-1} of $p_h$, we have
\begin{equation*}
\begin{aligned}
\small
\min_{v \in \St} \; \sdd u_h^+(x_i;v) &\leq
\min_{v \in \St} \; \sdd [\interp u - K_3 p_h](x_i;v) \\
&\leq C(d,\sigma) \delta_i^{k+\alpha-2} |u|_{\nCka(B_i)}
- K_3 \max_{v \in \St} \sdd p_h(x_i;v) \\
& \leq C(d,\sigma) \delta_i^{k+\alpha-2} |u|_{\nCka(B_i)}
- C(d,\sigma)|u|_{\nCkao} \; \delta_i^{k+\alpha-2} \le 0.
\end{aligned}
\end{equation*}
Therefore $\Tve[u_h^+;f](x_i) \leq \min_{v \in \St} \sdd u_h^+(x_i;v) \leq 0$.
The three cases show that $u_h^+$ is a discrete supersolution, and thus
by \Cref{L:DCP} (discrete comparison principle),
\begin{equation*}
\begin{aligned}
\uve \leq \; & \interp u - K_1 q_h + K_2 - K_3 p_h \\
= \; & \interp u + C(d,\sigma, \Omega) \frac{(\delta \theta)^{k+\alpha}
+ h^{k+\alpha}}{\delta^2} |u|_{\nCkao} \\
&+ C(d,\sigma) \delta^{k+\alpha} \l( |u|_{\nCkao} + |f|_{\nCkao} \r)
 + C(d,\sigma)|u|_{\nCkao} \delta^{k+\alpha}.
\end{aligned}
\end{equation*}
This, conjunction with the lower bound of $\uve$, completes
the proof.
\end{proof}

\begin{Corollary}[convergence rate]\label{C:convergence-rate}
Let $\Omega$ be strictly convex.
Let $u$ be the viscosity solution of \eqref{E:pde-CE} and $\uve$ be the discrete
solution of \eqref{E:2ScOp}. If $u \in \nCkao$ for $k=0,1$
and $0<\alpha \leq 1$, and $\theta \leq 1$, we have
\begin{equation}\label{E:convergence-rate}
    \|u-\uve\|_{L^\infty(\Oh)} \le C(\Omega,d,\sigma)
    \Big( |u|_{\nCkao} + |f|_{\nCkao} \Big) \; h^{\frac{(k+\alpha)^2}{2+k+\alpha}},
\end{equation}
provided $R_\alpha(u) :=
|u|_{\nCkao}^{\frac{1}{2+k+\alpha}}
\Big(|u|_{\nCkao} + |f|_{\nCkao} \Big)^{-\frac{1}{2+k+\alpha}}$
and
\begin{equation*}
\delta = R_\alpha(u) h^{\frac{k+\alpha}{2+k+\alpha}},
\quad
\theta = R_\alpha(u)^{-1} h^{\frac{2}{2+k+\alpha}}.
\end{equation*}
\end{Corollary}
\begin{proof}
Since the pointwise interpolation error satisfies \cite{BrennerScott}
\begin{equation*}
\|u - \interp u\|_{L^\infty(\Oh)} \le C h^{k+\alpha} |u|_{\nCkao}
\le C \frac{h^{k+\alpha}}{\delta^2} |u|_{\nCkao},
\end{equation*}
and $h \le \delta$, we end up with the error estimate
\begin{equation*}
\|u - \uve\|_{L^\infty(\Oh)} \le C \left[ |u|_{\nCkao}
\frac{h^{k+\alpha} + (\delta \theta)^{k+\alpha}}{\delta^2}
+ \Big(|u|_{\nCkao} + |u|_{\nCkao}\Big) \delta^{k+\alpha} \right].
\end{equation*}
In order to balance all contributions, we
first choose $\theta=\frac{h}{\delta}$ and next equate the two terms
on the right-hand side to obtain the asserted relations between
$\delta,\theta$ and $h$. This completes the proof.
\end{proof}

\begin{remark}[two important scenarios]\label{R:two-scenarios}
We want to point out two important scenarios based on the regularity of $u$ for \Cref{C:convergence-rate} (convergence rate).
\begin{enumerate}[$\bullet$]
\item Full regularity $u \in C^{1,1}(\Oc)$, 
i.e. $k = \alpha = 1$. The optimal choice of parameters $\delta \sim O(h^{1/2}), \theta \sim O(h^{1/2})$
in \Cref{C:convergence-rate} yields either a linear decay rate $O(h)$ or a quadratic rate $O(\delta^2)$ in terms of the fine scale $h$ or the coarse scale $\delta$.
\item Lipschitz regularity $u \in C^{0,1}(\Oc)$, i.e. $k = 0, \ \alpha = 1$.
Choosing optimal parameters $\delta \sim O(h^{1/3}), \theta \sim O(h^{2/3})$
in \Cref{C:convergence-rate} gives us either a rate $O(h^{1/3})$ in terms of the
fine scale $h$ or a linear rate $O(\delta)$ in terms of the coarse scale $\delta$.
\end{enumerate}
We point out that, since
$|u|_{C^{0,1}(\Oc)} \lesssim |f|_{C^{1,1}(\Oc)}$ and $|u|_{C^{1,1}(\Oc)} \lesssim |f|_{C^{3,1}(\Oc)}$ under proper assumptions of $\Omega$ \cite{DeFi}, 
the right hand side of \eqref{E:convergence-rate} can be bounded with only norms of $f$.
Our error estimates are thus realistic in terms of regularity.
\end{remark}

\begin{remark}[fine scale vs regularity]
It is instructive to realize that the coarse scale $\delta$ gets
finer with increasing regularity $k+\alpha$ of $u$, whereas the angular
scale $\theta$ gets coarser. This behavior is opposite to
the error estimates in \cite[Remark 5.4]{LiNo}.
\end{remark}

\begin{remark}[alternate proof]\label{R:alternate-proof-k0}
When $k=0$, the proof of \Cref{T:error-estimate} (error estimate) can be simplified a little bit. To be more specific, we can
construct a discrete supersolution $u_h^+ \in \Vh$ of
the form
\begin{equation*}
    u_h^+ = \interp u - K_1 q_h + K_2
\end{equation*}
provided that
\begin{equation*}
K_1 = C(d,\sigma) \frac{(\delta \theta)^{\alpha}
+ h^{\alpha}}{\delta^2} |u|_{\Czeroao}, \quad
K_2 = C(d,\sigma) \delta^{\alpha}
\l( |u|_{\Czeroao} + |f|_{\Czeroao} \r).
\end{equation*}
This is due to the fact that if $0 < \dist(x_i,\pO) < d\delta$,
then invoking \eqref{E:error-contact-0} with our choice of $K_2$
implies $\Tve[u_h^+;f](x_i) \leq 0$.
\end{remark}

\subsection{Non-attainment of Dirichlet condition}\label{S:non-attainment}
Although we mainly focus on the case that the domain $\Omega$ is strictly convex, it is also possible to modify and extend our two-scale method to compute the convex envelope over {\it convex polytopes} $\Omega$, thus domains with piecewise linear boundary. For simplicity, we only explain the ideas in ${\mR}^2$, but higher dimensions $d>2$ can be dealt with in a similar manner.

We need additional notation. A convex polytope $\Omega$ can be described by a set $\Nv$ of vertices on its boundary; thus $\Omega = \conv (\Nv)$. We then let $\Ne = \pO \setminus \Nv$ be the set of boundary edges of $\Omega$ excluding vertices. While $u = f$ is no longer true on $\pO$ if $\Omega$ is not strictly convex, it can be shown using \cite[Corollary 17.1.5]{Rockafellar2015convex} that $u = f$ at vertices of $\Nv$, and on each edge of $\Ne$, the function $u$ is the convex envelope of $f$ restricted to that edge. One can thus show that $u$ is the viscosity solution of the following fully nonlinear obstacle problem:
\begin{equation}\label{E:pde-CE-polytope}
\left\{
\begin{array}{ll}
T[u;f](x) = 0 \quad \; & \forall x \in \Omega, \\
\min\left\{f(x) - u(x), e^T(x) D^2u(x) e(x) \right\} = 0\quad \; & \forall x \in \Ne, \\
u(x) = f(x) \quad \; & \forall x \in \Nv,
\end{array}
\right.
\end{equation}
where $e(x)$ is a unit vector parallel to the edge of $\Omega$ containing $x \in \Ne$; note that \eqref{E:pde-CE-polytope} is a modification of \eqref{E:pde-CE} on $\partial\Omega$. To discretize this system, let $\Nhv := \Nv \subset \Nhb$ and $\Nhe := \Nhb \cap \Ne$, then our discrete problem is to find $\uve \in \Vh$ satisfying
\begin{equation}\label{E:2ScOp-Ex4}
\left\{
\begin{array}{ll}
\Tve[\uve;f](x_i) = 0 \quad \; & \forall x_i \in \Nhi,\\
\min\left\{f(x_i) - \uve(x_i), \sdd \uve(x_i, e(x_i)) \right\} = 0\quad \; & \forall x_i \in \Nhe,  \\
\uve(x_i) = f(x_i) \quad \; & \forall x_i \in \Nhv,
\end{array}
\right.
\end{equation}
where the step size of $\sdd \uve(x_i, e(x_i))$ should be defined as the maximum number $\delta_i$ in
$(0,\delta]$ such that $x_i \pm \delta_i e(x_i)$ are both inside $\Oc$. The convergence of $\uve$ can be derived in a similar way to \Cref{S:TwoSc}. We now prove an error estimate.

\begin{Proposition}[convergence rate for polytopes]\label{L:conv-rate-polytopes}
Let $\Omega$ be a convex polytope and $u \in \nCkao$ with $k=0,1, \ 0<\alpha \leq 1$, and $\theta \leq 1$. Let $\uve\in\Vh$ be the discrete solution of \eqref{E:2ScOp-Ex4}. If the discretization parameters $\ve = (h,\delta,\theta)$ obey relations similar to those in \Cref{C:convergence-rate} (convergence rate), then
\[
\| u - \uve \|_{L^\infty(\Omega)} \le C(u,\Omega,d,\sigma) \, h^{\frac{(k+\alpha)^2}{2+k+\alpha}}.
\]
\end{Proposition}
\begin{proof}
We first notice that $\Oh=\Omega$ and that \Cref{L:DCP} (discrete comparison principle) implies the following stability result: if $u_h, w_h \in \Vh$ satisfy
$\Tve[u_h;f](x_i) = \Tve[w_h;f](x_i)$ for all $x_i \in \Nhi$, then
\begin{equation}\label{E:stability}
    \max_{x_i \in \Nh} \left| u_h(x_i) - w_h(x_i) \right|
    \leq \max_{x_i \in \Nhb} \left| u_h(x_i) - w_h(x_i) \right|.
\end{equation}
We consider an auxiliary discrete problem: seek $\wt{u}_{\ve} \in \Vh$ that solves
\begin{equation*}
\left\{
\begin{array}{ll}
\Tve[\wt{u}_{\ve};f](x_i) = 0 \quad \; & \forall x_i \in \Nhi,\\
\wt{u}_{\ve}(x_i) = u(x_i) \quad \; & \forall x_i \in \Nhb.
\end{array}
\right.
\end{equation*}
We observe that \Cref{C:convergence-rate} 
still holds for $\wt{u}_{\ve}$, without the strict convexity assumption on $\Omega$, because the Dirichlet boundary is attained. Therefore, choosing $\delta$ and $\theta$ as in \Cref{C:convergence-rate}, we obtain
\begin{equation*}
    \Vert u - \wt{u}_{\ve} \Vert_{L^{\infty}(\Oh)} 
    \leq C(u,\Omega,d,\sigma) \, h^{\frac{(k+\alpha)^2}{2+k+\alpha}}.
\end{equation*}
It remains to estimate $\Vert \wt{u}_{\ve} - \uve \Vert_{L^{\infty}(\Oh)}$, for which
we resort to \eqref{E:stability} because both $\wt{u}_{\ve},\uve\in\Vh$.
Since the boundary subsystem
\begin{equation*}
\left\{
\begin{array}{ll}
\min\left\{f(x_i) - \uve(x_i), \sdd \uve(x_i, e(x_i)) \right\} = 0\quad \; & \forall x_i \in \Nhe,  \\
\uve(x_i) = f(x_i) \quad \; & \forall x_i \in \Nhv,
\end{array}
\right.
\end{equation*}
can be viewed as several one dimensional two-scale discretizations of the convex envelope problem, \Cref{C:convergence-rate} again implies
\begin{equation*}
\max_{x_i \in \Nhb} \left| \wt{u}_{\ve}(x_i) - \uve(x_i) \right| = 
\max_{x_i \in \Nhb} \left| u(x_i) - \uve(x_i) \right|
\leq C(u,\Omega,d,\sigma) \, h^{\frac{(k+\alpha)^2}{2+k+\alpha}}.
\end{equation*}
This concludes the proof.
\end{proof}

It is worth pointing out that we may not need a two-scale structure on the boundary since it reduces to a one dimensional problem on the edge of a polytope in 2D. However, notice that this procedure extends to dimensions $d>2$, and in such case boundary subproblems
possess dimension higher than one and require a two-scale structure.

\section{Modified Wide Stencil Method}\label{S:modified-wide-stencil}
Our numerical analysis of the previous sections
could be applied to derive error estimates for a modified
wide stencil method obtained upon adding a two-scale structure into that of \cite{Ob2}. 
Since key ideas and techniques are identical to those for the two-scale method, 
we present them without proofs. 
First let us briefly introduce the wide stencil method in a way convenient to our analysis; we refer the readers to \cite{Ob2} and \cite{ObRu} for more details.

For a strictly convex domain $\Omega \subset \mRd$, with abuse
of notations, let $\Nhi := \Omega \cap h\mZd$ be a Cartesian grid
in $\Omega$, and $\Vh$ be the space consisting of all maps
$u_h: \Nhi \cup \pO \rightarrow \mR$. Let a coarse scale
$\delta \ge \sqrt{d}{h}$ be used to define the set of discrete directions
\[
D_{\ve} := \l\{ x \in h\mZd: \dist\big(x, \partial B(0,\delta) \big) \le \frac{\sqrt{d}}{2}h \r\},
\]
where $\ve := (h,\delta)$ and $B(0,\delta)$ is the ball centered
at the origin with radius $\delta$. It is worth pointing out that $D_{\ve}$ is just a few layers of grid points, and thus its cardinality satisfies $\#D_{\ve} \lesssim \left(\frac{\delta}{h}\right)^{d-1}$.
The following lemma is similar to \cite[Lemma 4.4]{DolzWalk} and characterizes the consistency error
due to using $D_{\ve}$ instead of $\partial B(0,\delta)$.
 
\begin{Lemma}[properties of $D_{\ve}$]\label{L:consist-Dve}
For any $v \in \partial B(0,\delta)$, there exists $v_{\ve} \in D_{\ve}$ such that the angle between the vectors $v$ and $v_{\ve}$
is bounded by $\frac{\sqrt{d}\pi h}{4\delta}$. Moreover, $\frac{\delta}{2}\le |v| \le \frac{3\delta}{2}$ for all $v\in D_\ve$.
\end{Lemma}
\begin{proof}
Choose a Cartesian grid point in $v_{\ve} \in h\mZd$ closest to $v$, which in turn must satisfy $|v - v_{\ve}| \le \frac{\sqrt{d}h}{2}$, whence $v_{\ve} \in D_{\ve}$. The angle $\theta$ between $v$ and $v_{\ve}$ is dictated by \looseness=-1
\[
\sin\theta \le \frac{|v - v_{\ve}|}{\delta} \le \frac{\sqrt{d}}{2}\frac{h}{\delta}.
\]
This implies $\theta \leq \frac{\pi}{2}\sin\theta \le \frac{\sqrt{d}\pi h}{4\delta}$.
Moreover, by definition of $D_\ve$ we see that
$\frac{\delta}{2} \le \delta - \frac{\sqrt{d}}{2}h \le |v| \le
\delta + \frac{\sqrt{d}}{2}h \le \frac{3\delta}{2}$ for all $v\in D_\ve$.
\end{proof}

For any function $w \in \Vh$ and any vector $v \in D_{\ve}$,
let the centered second difference operator at any $x_i \in \Nhi$ in the direction $v$ be
\begin{equation*}
    \nabla^2_\ve w(x_i;v) := \frac{2}{\l(\rho_{+} + \rho_{-}\r)|v|^2}
    \l( \frac{w(x_i + \rho_{+}v) - w(x_i)}{\rho_{+}} + \frac{w(x_i - \rho_{-}v) - w(x_i)}{\rho_{-}} \r),
\end{equation*}
where $\rho_{\pm}$ are the biggest numbers in $(0, 1]$ such that $x_i \pm \rho_{\pm} v \in \Oc$. Notice that this is well-defined for any $w \in \Vh$ because $x_i \pm \rho_{\pm} v$ are either in $\Nhi$ or on the boundary $\pO$. Since for any
$v \in D_{\ve}$ we have $\frac{\delta}{2} \le |v| \le \frac{3\delta}{2}$, the parameter
$\delta$ plays a role similar to the coarse scale $\delta$ 
for second differences in our two-scale method.
The cardinalities
  $\#D_{\ve} \approx (\delta/h)^{d-1}$ and $\#\St \approx \theta^{-(d-1)}$ are consistent
provided $\theta \approx h/\delta$.

We define the discrete operator for the modified wide stencil method to be
\begin{equation*}
    \Tve[w;f](x_i) := \min\l\{f(x_i) - w(x_i),
    \min_{v \in D_{\ve}} \nabla^2_\ve w(x_i;v) \r\} \quad\forall \,  x_i \in \Nhi
\end{equation*}
for any $w \in \Vh$. Finally, the discrete problem reads: find $\uve \in \Vh$
such that
\begin{equation}\label{E:2ScOp-WD}
    \Tve[\uve;f](x_i) = 0 \quad\forall \,  x_i \in \Nhi,
\end{equation}
and $\uve(x) = f(x)$ for any $x \in \pO$.
It is now easy to check that \Cref{L:DCP} (discrete comparison principle) and
\Cref{Prop:Consistency-lowreg} (consistency for $u$ with H\"older regularity) are
valid verbatim in the present context, except that instead of \eqref{E:error-noncontact}
we now have
\[
\min_{v \in D_\ve} \nabla^2_\ve w(x_i;v) \le C(d,\sigma) \frac{h^{k+\alpha}}{\delta^2}
  |u|_{C^{k,\alpha}(B_i)}.
\]
In fact, the modified wide stencil method can be viewed as a modified version of two-scale method without interpolation error and $\theta \approx h/\delta$.

The following error estimate mimics that in \Cref{S:RatesHolder}. It is
a consequence of the discrete comparison principle and consistency for the
wide stencil method together with
the discrete barrier functions of \Cref{S:DBarrier}. We omit its proof.

\begin{Theorem}[error estimate for the wide stencil method]\label{T:error-estimate-WD}
Let $\Omega$ be strictly convex.
Let $u$ be the viscosity solution of \eqref{E:pde-CE} and $\uve$ be the discrete
solution of \eqref{E:2ScOp-WD}. If $u \in \nCkao$ for $k=0,1$
and $0<\alpha \leq 1$, then the following error estimate holds
\begin{equation*}\label{E:error-estimate-WD}
\left| u(x_i) - \uve(x_i) \right|
\leq C \left( |u|_{\nCkao} \frac{h^{k+\alpha}+ \delta^{2+k+\alpha}}{\delta^2}
+ |f|_{\nCkao} \delta^{k+\alpha} \right) \quad \forall x_i \in \Nhi,
\end{equation*}
with $C = C(\Omega,d,\sigma)$. If $\delta :=
|u|_{\nCkao}^{\frac{1}{2+k+\alpha}}
\Big(|u|_{\nCkao} + |f|_{\nCkao} \Big)^{-\frac{1}{2+k+\alpha}} h^{\frac{k+\alpha}{2+k+\alpha}}$,
we thus obtain the convergence rate
\begin{equation*}\label{E:convergence-rate-WD}
\left| u(x_i) - \uve(x_i) \right| \leq C(\Omega,d,\sigma)
    \Big( |u|_{\nCkao} + |f|_{\nCkao} \Big) \; h^{\frac{(k+\alpha)^2}{2+k+\alpha}} \quad \forall x_i \in \Nhi.
\end{equation*}
\end{Theorem}
  
We point out that \Cref{R:two-scenarios} (two important scenarios) applies in this context. In particular, the convergence rate is of order $O(h)$ provided $\delta = O(h^{1/2})$ for functions $u \in C^{1,1}(\Oc)$.

\section{Numerical Experiments}\label{S:Exp}

To solve the discrete system \eqref{E:2ScOp}, we use
Howard's algorithm which converges superlinearly.
We implemented the 2-scale method within MATLAB, using
some of the routines provided by the software FELICITY \cite{Walker1, Walker2}.

\subsection{Howard's Algorithm}\label{S:Howard-Algorithm}
For convenience, let us order the nodes in
$\Nh = \{ x_1, \ldots, x_N\}$ with $x_i \in \Nhi$
for $1 \leq i \leq N_0$ and $x_i \in \Nhb$ for
$N_0+1 \leq i \leq N$; thus $N, N_0$ and $N_b := N-N_0$
are the cardinality of $\Nh, \Nhi$ and $\Nhb$ respectively.
In addition, let
$\bu := (u_h(x_i))_{i=1}^N \in \mRN$ stand for the vector of
nodal values of a generic $u_h \in \Vh$, and
$\St = \l\{v_1,\ldots,v_{S} \r\}$, where $S$ is
the cardinality of $\St$. In view of the expression \eqref{E:disc-oper} for the discrete operator $\Tve$, the discrete system \eqref{E:2ScOp} reads
\begin{equation}\label{E:Howard-disc-system}
\sup_{\ba \in \mA}
\l( B^{\ba} \bu - F^{\ba} \r) = \bm{0},
\end{equation}
where $\mA = \l\{(\alpha_1,\ldots,\alpha_{N_0}): \alpha_i \in \{j\}_{j=0}^S \r\}$,
matrix $B^{\ba}\in\mathbb{R}^{N\times N}$ satisfies
\begin{equation*}
\l( B^{\ba} \bu \r)_i = \l\{
\begin{array}{ll}
u_h(x_i) \quad & i \geq N_0+1,\; 0 \le \alpha_i \le S \\
u_h(x_i) \quad & 1 \leq i \leq N_0,\;\alpha_i = 0,  \\
-\sdd u_h(x_i;v_{\alpha_i}) \quad & 1 \leq i \leq N_0,\;1 \le \alpha_i \le S,
\end{array}
\r.
\end{equation*}
and $F^{\ba}$ is given by
\begin{equation*}
\l( F^{\ba} \r)_i = \l\{
\begin{array}{ll}
f(x_i) \quad & i \geq N_0+1, \;0 \le \alpha_i \le S \\
f(x_i) \quad & 1 \leq i \leq N_0,\;\alpha_i = 0,  \\
0 \quad & 1 \leq i \leq N_0,\; 1 \le \alpha_i \le S.
\end{array}
\r.
\end{equation*}
We solve \eqref{E:Howard-disc-system} via the Howard's algorithm \cite{BoMaZi}, which is a semi-smooth Newton method \cite{BoMaZi,HIK:2002,SmearsSuli2014,Ulbrich:2011} also known as policy iteration in the financial literature \cite{PutermanBrumelle1979}:
\begin{algorithm}
  \caption{(Howard's Algorithm)
    \label{alg:Howard}}
  \begin{algorithmic}[1]
    \State Select an arbitrary initial $\ba_0 \in \mA$, and let $n=0$.
    \While{}
        \State Let $\bu_{n}$ be the solution of the linear equations
        $B^{\ba_n} \bu_{n} - F^{\ba_n} = \bm{0}$.
        \State Let $\ba_{n+1} = \argmax_{\ba \in \mA} \l( B^{\ba} \bu_{n} - F^{\ba} \r)$.
        \State If $\ba_{n+1} = \ba_n$, stop; else $n = n+1$.
    \EndWhile
  \end{algorithmic}
\end{algorithm}

\noindent
Hereafter, the vector equality in \eqref{E:Howard-disc-system} and inequalities $\ge$ later are understood componentwise.
We could immediately see from the above
that for any $\ba \in \mA$,
we have $\l(B^{\alpha}\r)_{ii} > 0$
and $\l(B^{\alpha}\r)_{ij} \leq 0$ for $i \neq j$.
In fact, we prove that $B^{\ba}$ is an M-matrix.

\begin{Lemma}[M-matrix property]\label{L:M-Matrix}
For any $\ba \in \mA$, $B^{\ba}$ is an M-matrix.
\end{Lemma}
\begin{proof}
We only need to prove $B^{\ba} \bu \geq \bm{0}$
implies $\bu \geq \bm{0}$. 
Given two vectors $\bu,\bw \in \mR^{N}$ so that $B^{\ba} \bu \ge B^{\ba} \bw$ for all $\ba \in \mA$, we deduce $u_h \geq w_h$ for the corresponding functions $u_h,w_h\in\Vh$ in view of \Cref{L:DCP} (discrete comparison principle). This immediately implies $\bu\ge\bw$, and, upon taking $\bw=\bm{0}$, that $\bu\ge\bm{0}$ as desired.
\end{proof}

Invoking the fact that $B^{\ba}$ is an M-matrix
and applying \cite[Theorem 2.1]{BoMaZi}, we deduce that the $n$-th iterate
$\bu_n$ of Howard's algorithm converges monotonically
and superlinearly to $\uve$ as $n \to \infty$. The latter follows from the semi-smooth
Newton structure of Algorithm \ref{alg:Howard}. The former is a consequence of its
step 4 because
\[
B^{\ba_{n+1}} \bu_{n} - F^{\ba_{n+1}} \ge 
B^{\ba_{n}} \bu_{n} - F^{\ba_{n}} = \bm{0} = B^{\ba_{n+1}} \bu_{n+1} - F^{\ba_{n+1}} ,
\]
whence $\bu_{n+1} \le \bu_n$. Moreover, \cite[Theorem 2.1]{BoMaZi}
automatically gives existence and uniqueness
of our discrete system \eqref{E:2ScOp},
which we also proved in \Cref{L:Exist-Uniq-Stab} (existence, uniqueness and stability).
In practice, when
$\Vert \sup_{\ba \in \mA} \l( B^{\ba} \bu_n - F^{\ba} \r) \Vert_2$ is sufficiently small we can stop Algorithm \ref{alg:Howard}; we thus use the criterion
\begin{equation*}
\Vert \Tve[u_n;f] \Vert_{L^2(\Omega)} \leq 10^{-10} \Vert \Tve[f;f] \Vert_{L^2(\Omega)}
\end{equation*}
in all numerical experiments below.

\subsection{Accuracy}\label{S:Accuracy}
We now present several examples to examine the performance of
the two-scale method \eqref{E:2ScOp} for the convex envelope problem.
We choose $\delta = C_{\delta} h^{\alpha}$ and
$\theta = C_{\theta} h^{\beta}$ for different
$C_{\delta}, \alpha, C_{\theta}, \beta > 0$ in our experiments, and compare the computational rates with our theoretical rate of \Cref{C:convergence-rate} (convergence rate).

\begin{example}[full regularity $u \in C^{1,1}(\Oc)$] \label{Ex:ex1}
Let $\Omega = \{x\in \mathbb{R}^2: |x|<1 \}$ be the unit circle and
$f(\bm{x}) =  \cos(2 \pi |\bm{x}|)$. Then the convex envelope $u$
is given by
\begin{equation*}
u(x) = \left\{
\begin{array}{ll}
0 \;, &  \text{if} \quad |x| \leq 0.5 \\
\cos \left(2 \pi |x| \right) \;,
& \text{if} \quad 0.5 < |x| \leq \alpha_{*} \\
\cos \left(2 \pi \alpha_* \right) -
2\pi \sin \left(2 \pi \alpha_* \right) \left( |x| - \alpha_* \right) \;,
& \text{if} \quad \alpha_{*} < |x| \leq 1,
\end{array}
\right.
\end{equation*}
where the constant $\alpha_{*} \approx 0.6290$ satisfies the equation
\begin{equation*}
\cos \left(2 \pi \alpha_* \right) -
2\pi \sin \left(2 \pi \alpha_* \right) \left( 1 - \alpha_* \right) = 1.
\end{equation*}
The contact set $\mC(f)$ consists of two disjoint sets $\{\frac12\le|x|\le\alpha_*\}$ and
$\partial\Omega$.

In this example we have $f$ smooth and $u \in C^{1,1}(\Oc)$ (full regularity).
Upon choosing $\delta = 0.5 h^{1/2}$ and $\theta\approx 0.25 h^{1/2}$ we obtain computationally a linear convergence rate with respect to $h$, thus consistent
with \Cref{C:convergence-rate} (convergence rate), and report it
in \Cref{Table:Ex1} and Figure \ref{F:Ex1}. Plots of $\uve$ and $f$ are shown in \Cref{F:Ex1-function-plot} and slices of these functions on $\{(x,0): x \ge 0\}$ are depicted in \Cref{F:Ex1} (left). In \Cref{F:Ex1} (right), we also display the $L^{\infty}$ error vs meshsize $h$ for several choices $\delta = O(h^{\alpha})$ with different values of $\alpha$ together with $\theta \approx 0.25h^{1/2}$. The convergence rate for $\delta = O(h^{2/3})$ is better than the one predicted in \Cref{C:convergence-rate}, but other rates are consistent with our theory. We choose $\theta$ to be small enough to make the error induced by $\theta$ small relative to those of $\delta$ and $h$. In fact, we can see from \Cref{F:Ex1} (right) that the effect of changing from $\theta \approx 0.25 h^{1/2}$ to $\theta \approx h^{1/2}$ is relatively small, and thus conclude that $\theta$ is not a sensitive parameter.
\begin{table}[h!]
\begin{center}
\begin{tabular}[t]{ | l | l | c | c |}
\hline
Degrees of freedom    & Number of directions & $L^{\infty}-$error  & Iteration steps \\
\hline\hline
$N = 1557$, $h=2^{-4}$	 &	  \qquad\quad $S = 26$         	&  $3.769 \times 10^{-2}$			& 6  \\
\hline
$N = 6317$, $h=2^{-5}$	 &	  \qquad\quad $S = 36$         	&  $1.887 \times 10^{-2}$			& 10  \\
\hline
$N = 25469$,  $h=2^{-6}$	  &   \qquad\quad $S = 51$	 & 	$9.617 \times 10^{-3}$			&  11 \\
\hline
$N = 102445$,  $h=2^{-7}$	&	\qquad\quad $S = 72$	           & 	$4.801 \times 10^{-3}$			&  11 \\
\hline
$N = 410793$,  $h=2^{-8} $  	&	\qquad\quad $S = 101$	           & 	$2.400 \times 10^{-3}$			&   11  \\
\hline
\end{tabular}
\end{center}
\vskip0.2cm
\caption{\small \Cref{Ex:ex1}: $\delta = 0.5h^{1/2},\theta \approx 0.25 h^{1/2}$.
The convergence rate is about linear
(see \Cref{F:Ex1}), thus consistent with \Cref{C:convergence-rate}.
The number of search directions $S$ scales like $S\approx\theta^{-1}\approx h^{-1/2}$,
whereas the number of Howard's steps is relatively uniform.}
\label{Table:Ex1}
\end{table}
\begin{figure}[!htb]
\includegraphics[width=0.48\linewidth]{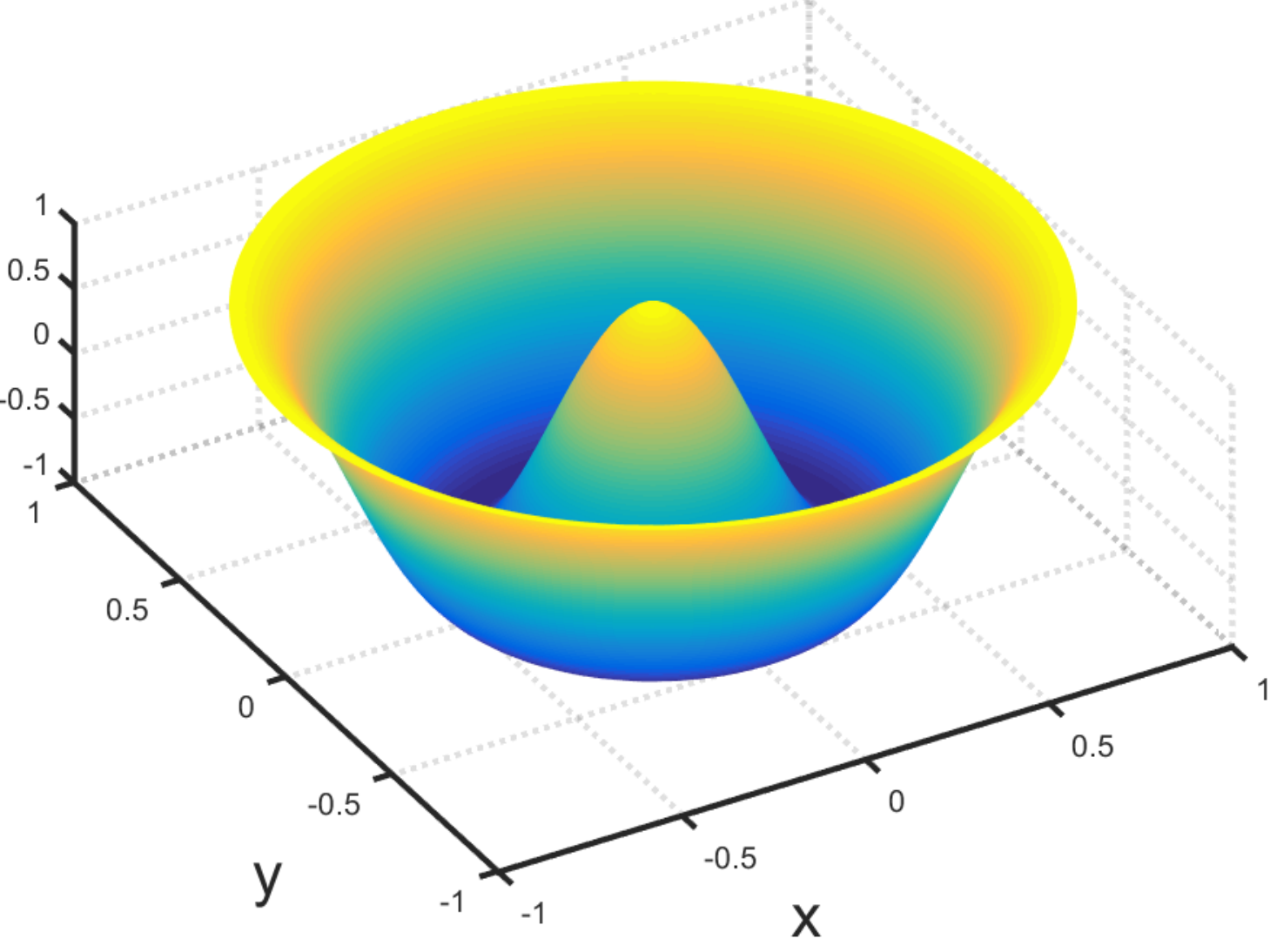}
\includegraphics[width=0.48\linewidth]{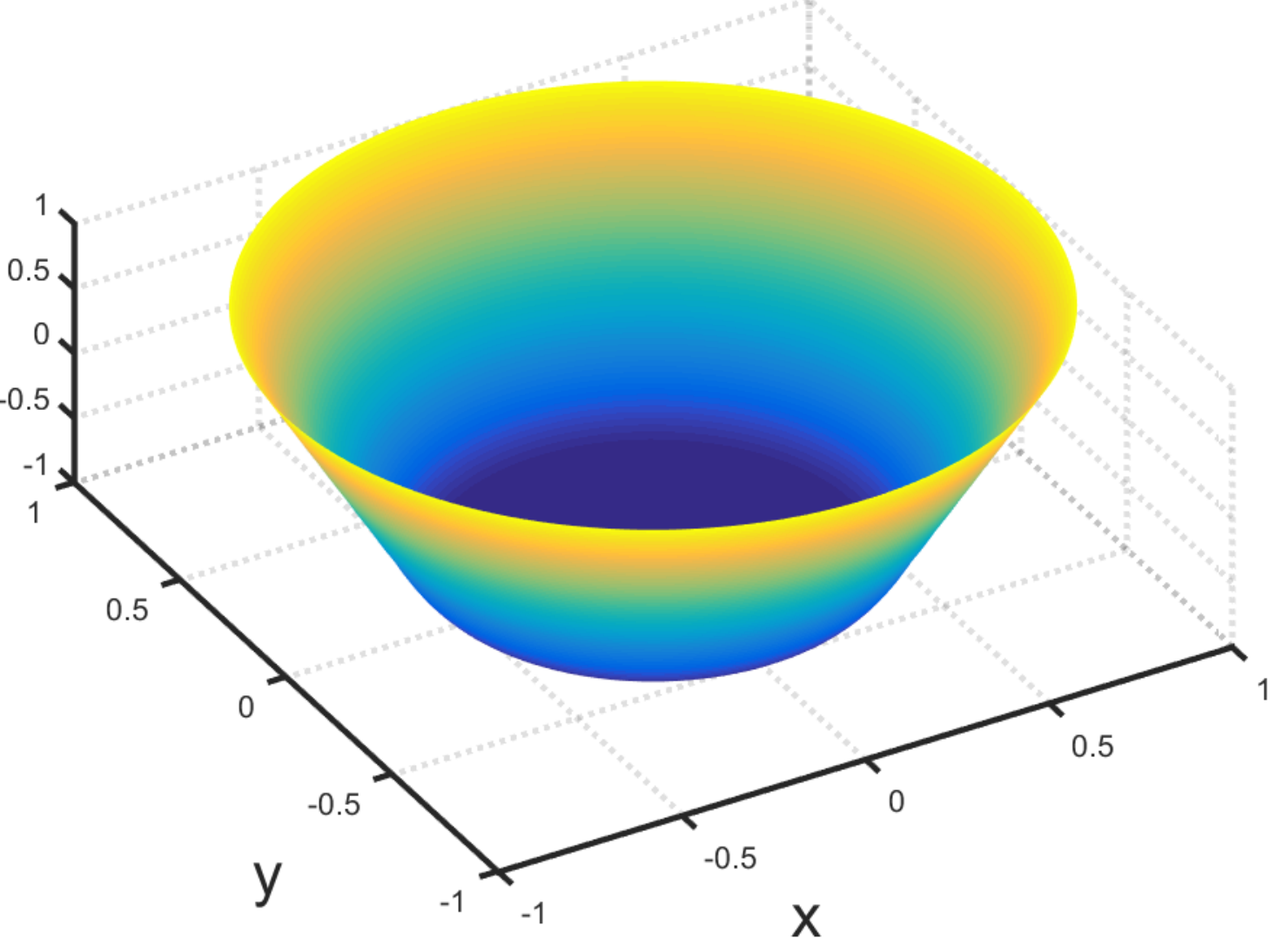}
\caption{\small \Cref{Ex:ex1},
left: plot of $f$; right: plot of $\uve$ for $h = 2^{-6}$.
}
\label{F:Ex1-function-plot}
\end{figure}

\begin{figure}[!htb]
\includegraphics[width=0.48\linewidth]{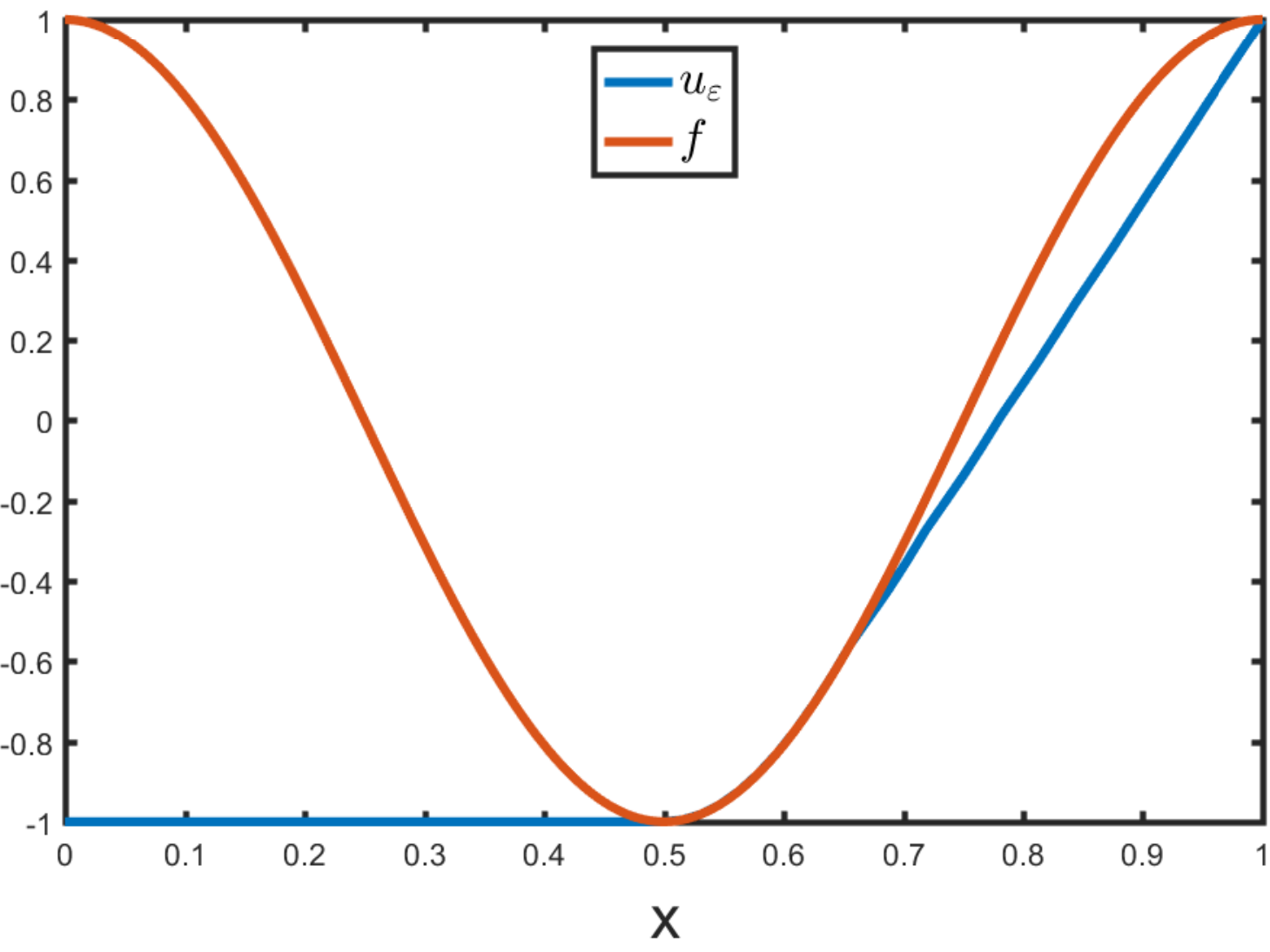}
\includegraphics[width=0.48\linewidth]{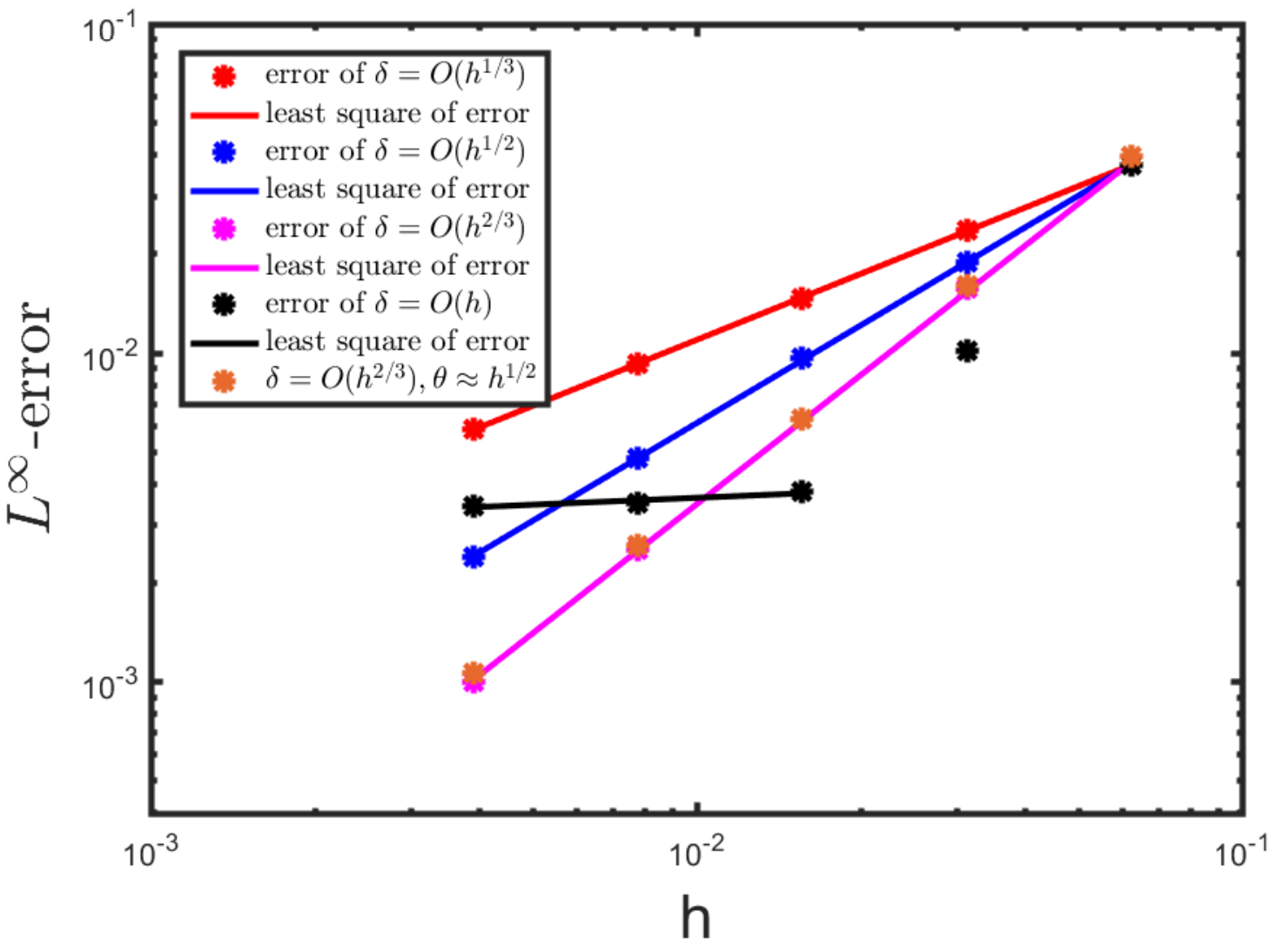}
\caption{\small
\Cref{Ex:ex1}. Left: slice of numerical solution $\uve$ on $\{(x,0): x \ge 0\}$ with $h = 2^{-6}, \delta = 0.25h^{1/2}, \theta \approx 0.25h^{1/2}$. Right: experimental rates of convergence
upon choosing $\theta \approx 0.25h^{1/2}$ and $\delta = O(h^{\alpha})$ with $\alpha = 1/3, 1/2, 2/3, 1$. A least square regression is performed for $h^{-k}$ with $k=6,7,8$ and
the case $\delta = O(h)$. The orders are about $0.67, 0.99, 1.30, 0.07$. We also plot the errors for $\theta \approx h^{1/2}, \delta =h^{2/3}$, and the errors are very close to choosing $\theta \approx 0.25h^{1/2}, \delta =h^{2/3}$.}
\label{F:Ex1}
\end{figure}
\end{example}

\begin{example}[Lipschitz regularity $u \in C^{0,1}(\Oc)$] \label{Ex:ex2}
Let $\Omega = \{x\in\mathbb{R}^2: |x|<1 \}$ and
\begin{equation*}
f(x) = \left\{
\begin{array}{ll}
1-4|x|, & 0 \leq |x| < 1/4 \\
4|x|-1, & 1/4 \leq |x| < 1/2 \\
2-2|x|, & 1/2 \leq |x| < 3/4 \\
2|x|-1, & 3/4 \leq |x| \leq 1,
\end{array}
\right. \qquad
u(x) = \left\{
\begin{array}{ll}
0, & 0 \leq |x| < 1/4 \\
|x|-1/4, & 1/4 \leq |x| < 3/4 \\
2|x|-1, & 3/4 \leq |x| \leq 1.
\end{array}
\right.
\end{equation*}
This example deals with $f, u \in C^{0,1}(\Oc)$, i.e.
both $f$ and $u$ are Lipschitz. The contact set $\mC(f)$ consists of two disjoint components
$\{x\in \mathbb{R}^2: |x| \ge 3/4 \}$ and $\{x\in \mathbb{R}^2: |x|=1/4 \}$. See \Cref{F:Ex2} (left) that displays
slices on $\{(x,0): 0\le x\le 1\}$ of $f,u$ and the numerical solution
$\uve$ with $h = 2^{-6}, \delta = 0.25 h^{1/2}, \theta \approx 0.25h^{1/2}$.
We point out that the pointwise error is very small in the regions
$\{x\in \mathbb{R}^2: |x| \ge 3/4 \}$ and $\{x\in \mathbb{R}^2: |x| \le 1/4 \}$; in the latter $u$ is linear and thus the interpolation error disappears. On the other hand, in the region $\{x\in \mathbb{R}^2: 1/4 < |x| < 3/4 \}$, where
$u$ is only linear in the radial direction, we observe larger error
for $\uve$.
Experimental convergence rates for different choices of $\delta = O(h^{\alpha})$ are plotted in
\Cref{F:Ex2} (right): we see that these rates are better than those predicted in \Cref{C:convergence-rate} (convergence rate).
This theoretical rate can be improved upon exploiting that 
both functions $f$ and $u$ are non-smooth only at $\{0\}$ and across the curves
$\{|x|=1/4\}$ and $\{|x|=3/4\}$. In fact, for those $x_i \in \Nhi$ satisfying $\big| |x_i|-1/4 \big| \leq \delta$ or $\big| |x_i|-3/4 \big| \leq \delta$, according to \Cref{Prop:Consistency-lowreg} (consistency for $u$ with H\"older regularity), we have
\begin{equation*}
\Tve[\interp u;f](x_i) \leq f(x_i) - u(x_i) \leq C(u) \delta,
\end{equation*}
whereas for the rest of $x_i \in \Nhi$ the consistency error can be estimated
exactly as for $f, u \in C^{1,1}(\Oc)$. Therefore carrying out the
same analysis as in \Cref{T:error-estimate} (error estimate), we end up with the
error estimate
\begin{equation*}
\|u-\uve\|_{L^\infty(\Oh)} \le C(u)
\l( \delta + \frac{(\delta \theta)^2 + h^2}{\delta^2}\r).
\end{equation*}
This yields a rate $O(h^{2/3})$ provided $\delta = O(h^{2/3})$, which is twice better than the rate from \Cref{C:convergence-rate} but still worse than the experimental ones in \Cref{F:Ex2} (right).

\begin{figure}[!htb]
\includegraphics[width=0.48\linewidth]{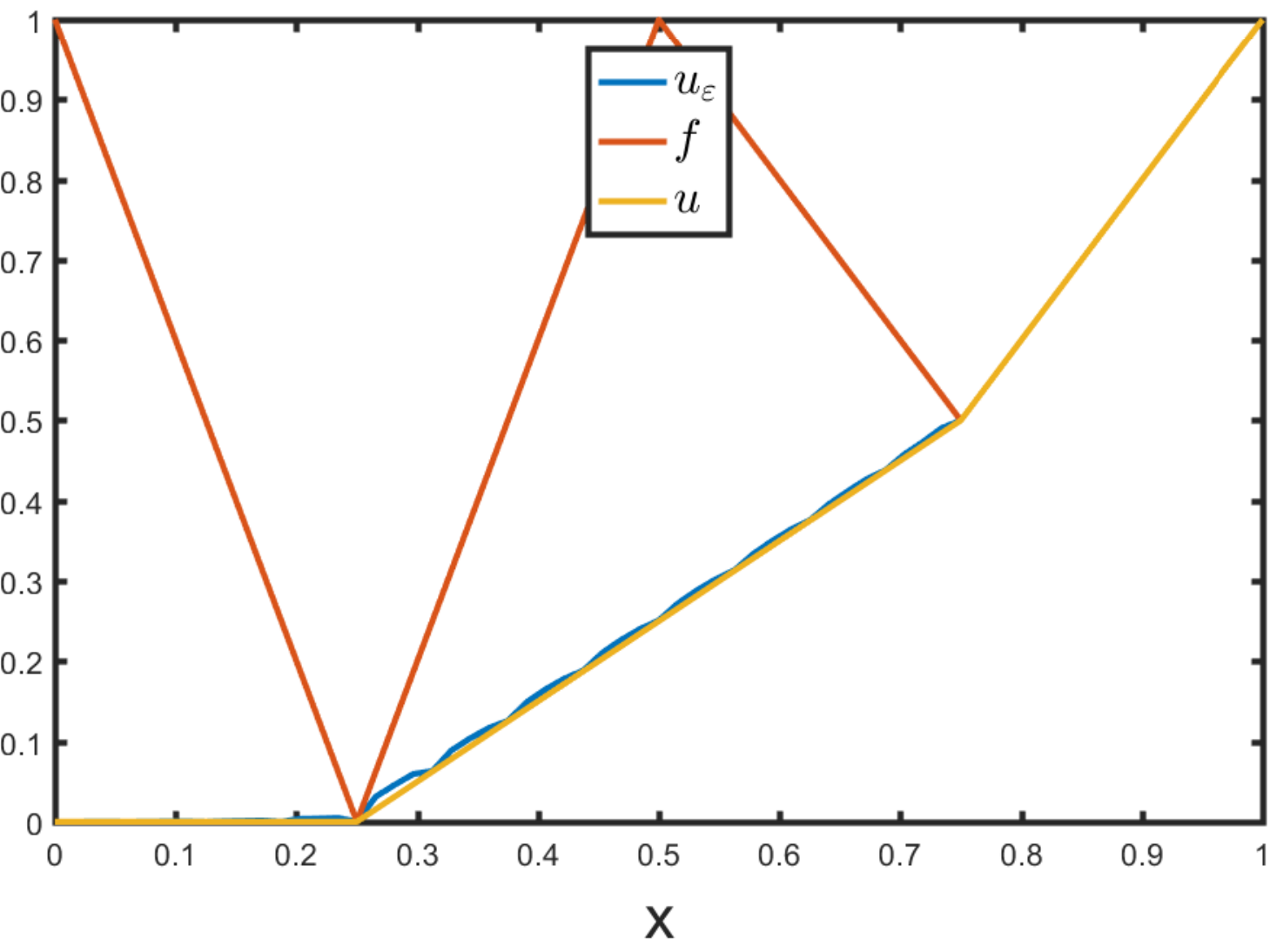}
\includegraphics[width=0.48\linewidth]{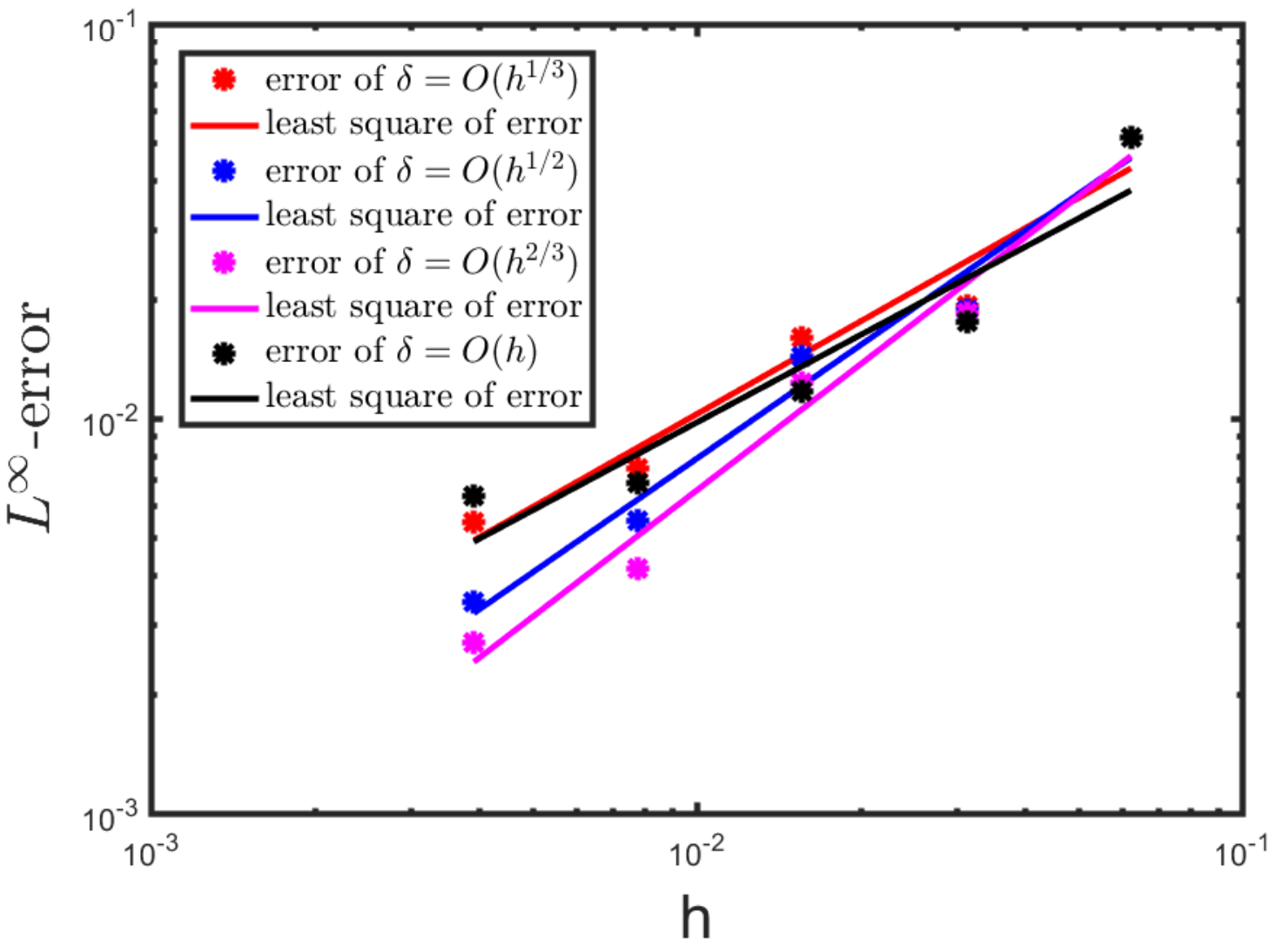}
\caption{\small \Cref{Ex:ex2}. Left: slices of $f,u$ and numerical solution $\uve$ on $\{(x,0): x \ge 0\}$ with $h = 2^{-6}, \delta = 0.25h^{1/2}, \theta \approx 0.25h^{1/2}$. Right: experimental rates of convergence upon choosing $\theta = O(h^{1/2})$ and $\delta = O(h^{\alpha})$ with $\alpha = 1/3, 1/2, 2/3, 1$. The orders are about $0.78, 0.96, 1.06, 0.74$.}
\label{F:Ex2}
\end{figure}
\end{example}

\begin{example}[Lipschitz $u \in C^{0,1}{(\overline{\Omega})}$ and nonstrictly convex $\Omega$]\label{Ex:ex3} Let $\Omega = (-1,1)^2$ and $f,u$ be as in \cite[Example 6.3]{Ob2} with $\alpha = \beta = 1$, i.e.
\[
f(x,y) = xy \;,\qquad u(x,y) = |x+y| - 1 .
\]
We point out that the Dirichlet boundary condition $u = f$ is attained on $\pO$ although the
domain $\Omega$ is not strictly convex, whence  \Cref{T:error-estimate} (error estimates) still applies.
In this example, $f$ is smooth but $u$ is only Lipschitz because $\Omega$ is not uniformly convex and non-smooth: $u$ exhibits a kink across the diagonal $\{(x,y): x+y=0\}$ and is piecewise linear otherwise. Moreover, $u<f$ in $\Omega$ whence the contact set $\mC(f)$ reduces to
$\partial\Omega$.

\Cref{F:Ex3} (left) displays slices on $\{(x,y): x \ge 0, \ y = x\}$ of $f,u$ and the numerical solution $\uve$ with $h = 2^{-6}, \delta = h^{1/2}, \theta \approx 0.25h^{1/2}$. One can observe a clear mismatch between $\uve$ and $u$ near the singular set $\{(x,y): x+y=0 \}$. Compared with \Cref{Ex:ex1} (full regularity $u\in C^{1,1}(\overline\Omega)$), the lack of regularity of $u$ here entails larger consistency error and $L^{\infty}$ error between $\uve$ and $u$. Experimental convergence rates for different choices of $\delta = O(h^{\alpha})$ are depicted in \Cref{F:Ex3} (right); we see that the best convergence rate $O(h^{0.58})$ is found when $\delta = O(h^{1/3})$,
which is again better than the $O(h^{1/3})$ rate predicted in \Cref{C:convergence-rate} (convergence rate).

\begin{figure}[!htb]
\includegraphics[width=0.495\linewidth]{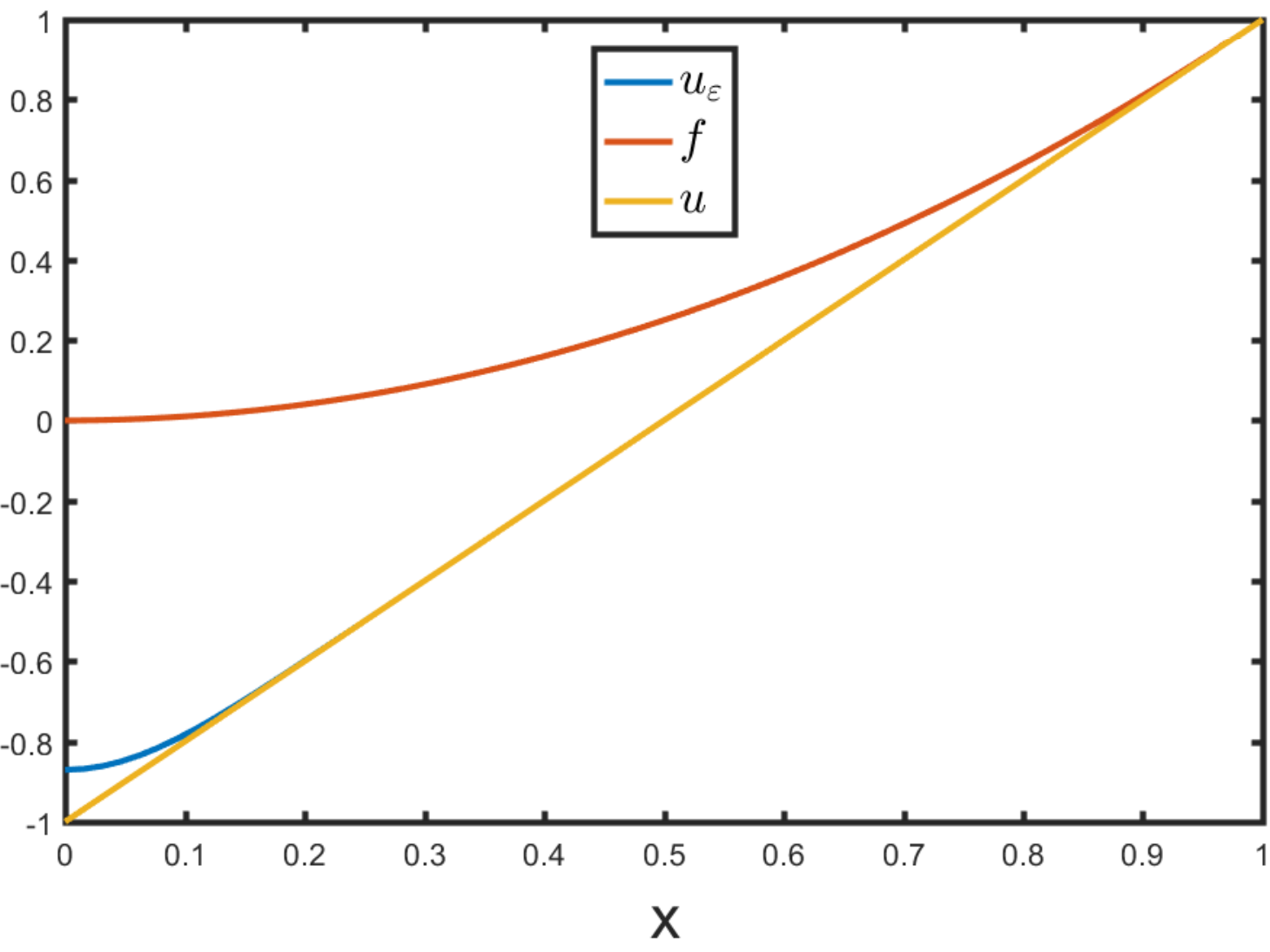}
\includegraphics[width=0.495\linewidth]{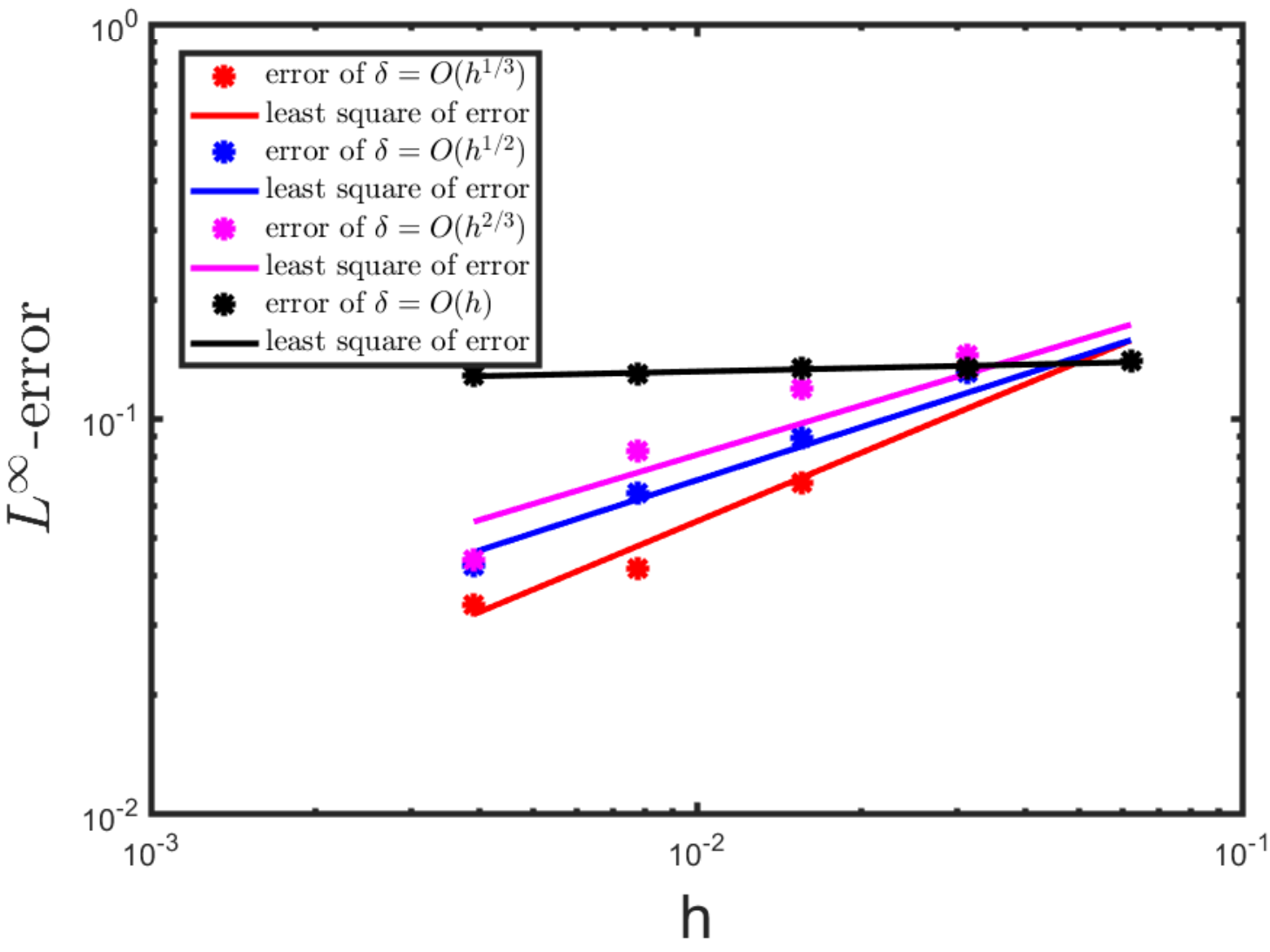}
\caption{\small \Cref{Ex:ex3}. Left: slice of numerical solution $\uve$ on $\{(x,y): x \ge 0, \ y = x\}$ with $h = 2^{-6}, \delta = h^{1/2}, \theta \approx 0.25h^{1/2}$. Right: experimental rates of convergence upon choosing $\theta = O(h^{1/2})$ and $\delta = O(h^{\alpha})$ with $\alpha = 1/3, 1/2, 2/3, 1$. The orders are about $0.58, 0.45, 0.41, 0.03$.}
\label{F:Ex3}
\end{figure}
 
\end{example}

\begin{example}[non-attainment of Dirichlet condition] \label{Ex:ex4}
Let $\Omega = (-1,1)^2$ and the function $f$ be $f(x,y) = \cos(\pi x)\cos(\pi y)$, whose restriction to $\pO$ is not convex. According to our definition \eqref{E:def-CE},
the convex envelope is given by
\begin{equation*}
u(x,y) =
  \begin{cases}
  -1    & \quad |x|+|y| \leq 1 \\
  -\cos \big(\pi (|x|+|y| - 1) \big) 
  &  \quad 1 < |x|+|y| \leq 1 + \beta_{*} \\
  -\cos \left(\pi \beta_* \right) +
  \pi \sin \left(\pi \beta_* \right) \big( |x|+|y| - 1 - \beta_* \big) 
  & \quad 1+\beta_{*} < |x|+|y|,
  \end{cases}
\end{equation*}  
where the constant $\beta_{*} \approx 0.2580$ satisfies the equation
\[
-\cos(\pi \beta_*) + \pi \sin(\pi \beta_*)(1 - \beta_*) = 1.
\]
This assertion requires a brief explanation. First
of all note that by symmetry it suffices to examine the first quadrant $0\le x,y \le 1$.
On the edges $\{y=1\}$ and $\{x=1\}$ the function $u$ is convex by construction and definition of $\beta_*$; see \Cref{F:Ex4} (left). Since $u$ is flat along lines $x+y=\beta$ and convex along perpendicular lines, we infer that $u$ is convex. It remains to show that $u\le f$ and $\ge$ than the convex envelope. To this end, we take convex combinations of boundary values $u(\beta-1,1)$ and $u(1,\beta-1)$ along the line $x+y=\beta$ with $1\le\beta\le2$ and show
that they are $\le f(x,y)$. For $\beta=1$ we realize that $u(x,y)=-1\le f(x,y)$ on $x+y=1$ and by symmetry for all $x+y\le1$. For $\beta>1$ a tedious calculation gives $u(x,y)=u(\beta-1,1) \le f(\beta-1,1) \le f(x,y)$ along $x+y=\beta$ as desired. We finally point out that the contact set $\mC(f)$ consists of four boundary segments of length $2\beta_*$
  centered at $(0,\pm 1), (\pm 1,0)$ and the four vertices $(\pm 1,\pm 1)$
  of $\Omega$; see \Cref{F:Ex4} (left). \looseness=-1

\begin{figure}[!htb]
\includegraphics[width=0.48\linewidth]{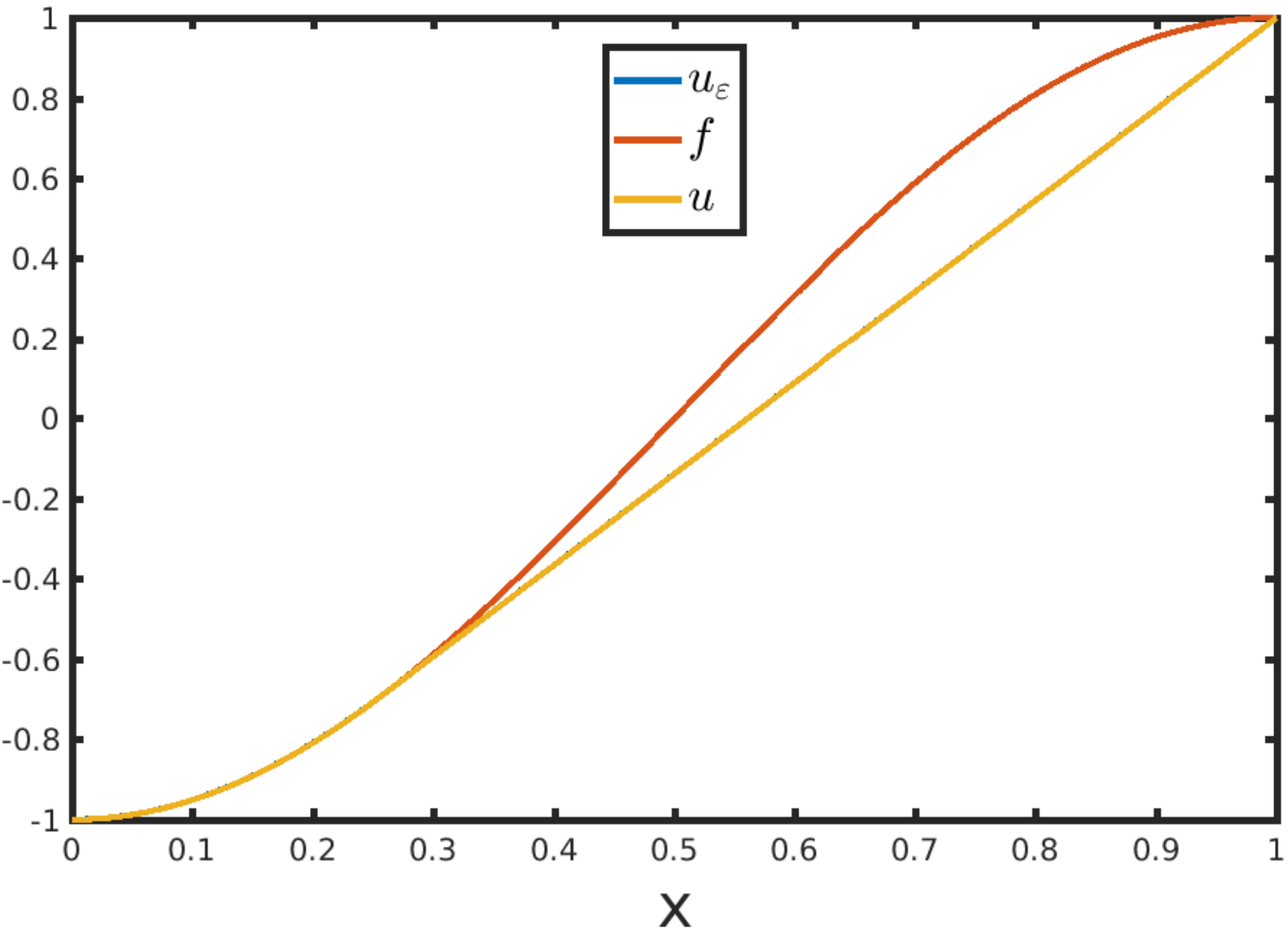}
\includegraphics[width=0.48\linewidth]{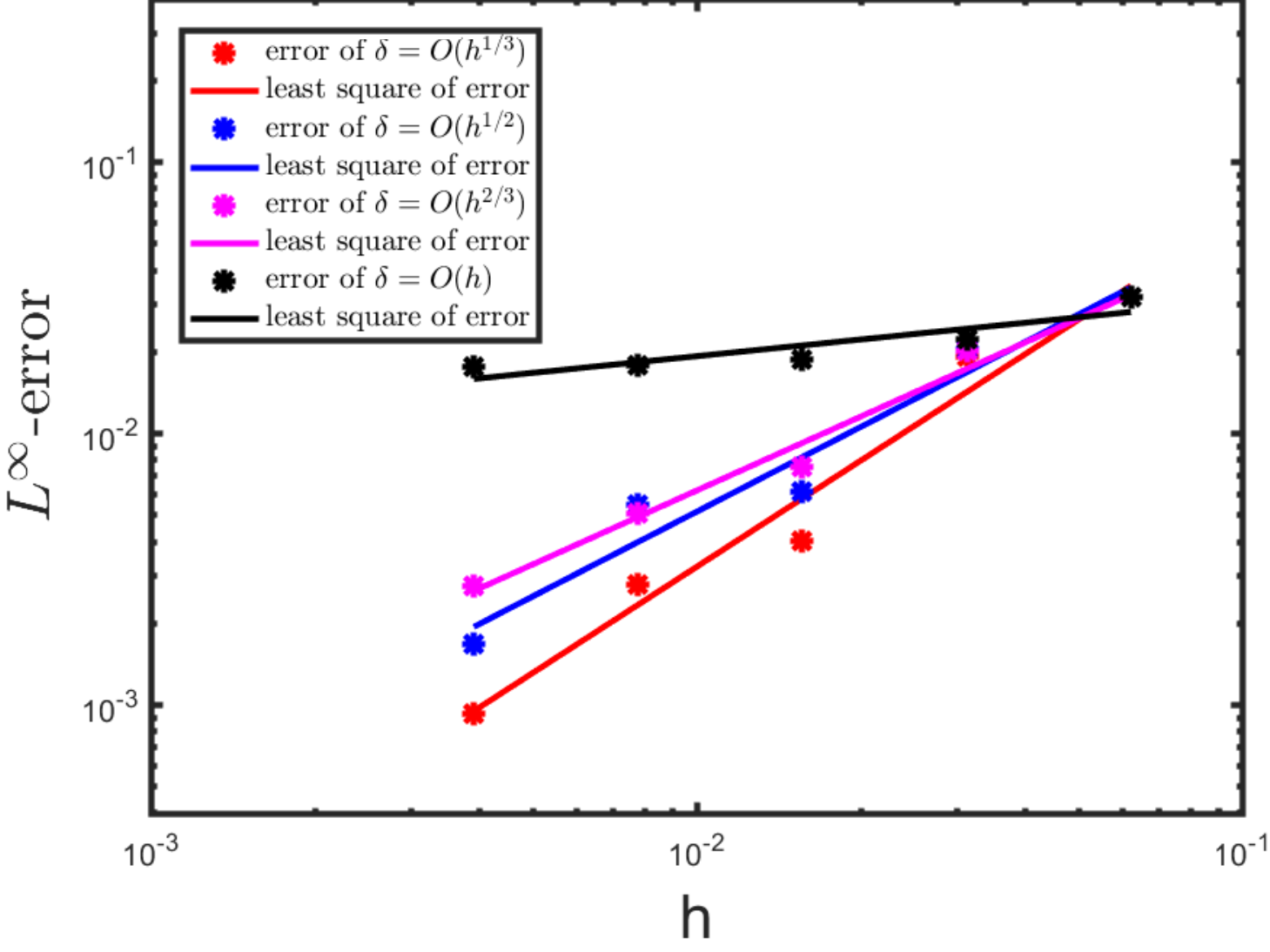}
\caption{\small \Cref{Ex:ex4}. Left: slices $f,u$ and
$\uve$ on the set $\{(x,1): x \ge 0 \}$ with $h = 2^{-6}, \delta = 2h^{1/2}, \theta \approx 0.5h^{1/2}$. Note that $\uve$ is indistinguishable from $u$ on this part of $\pO$. Right: experimental rates of convergence upon choosing $\theta = O(h^{1/2})$ and $\delta = O(h^{\alpha})$ with $\alpha = 1/3, 1/2, 2/3, 1$; the orders of convergence are about $1.30, 1.04, 0.91, 0.20$.}
\label{F:Ex4}
\end{figure}
We implemented the modified two-scale method \eqref{E:2ScOp-Ex4}, which first solves boundary subproblems on each edge of $\partial\Omega$ to find the trace of the discrete convex envelope $\uve$ and next determines $\uve$ within $\Omega$. \Cref{F:Ex4} (left) shows $f,u$
and $\uve$ on the boundary set $\{(x,1): 0\le x \le 1\}$; we point out that $u(x,1)=f(x,1)$ for $|x|\le \beta_*$. \Cref{F:Ex4} (right) displays the $L^{\infty}$ error for several choices of $h$ and $\delta$: we see that the experimental convergence rate is about $O(h)$ for $\delta=O(h^{1/2})$, in agreement with theory, but the rates for $O(h^{\alpha})$ with $\alpha = 1/3, 2/3$ seem to be better than those predicted in \Cref{C:convergence-rate} (convergence rate).
\end{example}

\subsection{Computational performance}
Thanks to the search tools provided by FELICITY \cite{Walker1,Walker2}, the process of locating the triangle of the mesh containing points $x_i \pm \delta_i v_j$ and computing the barycentric coordinates only takes a small percentage of the total computing time; this is consistent with the two-scale method for the \MA equation in \cite{NoNtZh1}. In \Cref{Ex:ex1} for $h = 2^{-6}, \delta = 0.25h^{1/2}, \theta \approx 2h^{1/2}$, this process is 6.7\% ($<$ 4 sec) of the total computation time (56.2 sec).
The most time consuming part of the experiment is constructing and solving the linear systems, i.e. the third line in \Cref{alg:Howard}; this takes 53.2\% of the total time. We do not attempt to exploit the sparsity pattern of the matrix $B^{\ba}$ and simply resort to MATLAB backslash command for solving linear systems; we leave this important issue open. All of our computations are performed on an Intel Xeon E5-2630 v2 CPU (2.6 GHz), 16 GB RAM using MATLAB R2016b.

\subsection{Comparison with other existing methods}
In this subsection, we briefly compare our two-scale method with two other methods for the computation of convex envelopes: the wide stencil method in \cite{Ob2} and the modified version of Dolzmann's method in \cite{Bartels}. Both the wide stencil method and our two-scale method are derived from the PDE formulation \eqref{E:pde-CE}, and have a discrete operator with similar structure. As explained in \Cref{S:modified-wide-stencil}, the wide stencil method can be viewed as a two-scale method with no interpolation error but with the constraint
$\theta \approx h/\delta$. Our two-scale method suffers from the interpolation error but allows some freedom in the choice of parameters and works well on unstructured grids, which provide geometric flexibility to fit the boundary $\pO$.

The modified version of Dolzmann's method in \cite{Bartels}, built for the computation of rank-one convex envelopes of functions defined on $\mathbb{R}^{n \times m}$, can be applied to compute the convex envelope by simply letting $m = 1$. When applied to compute convex envelopes, the technique of \cite{Bartels} hinges on the following algorithm: if $f^{(0)} = f$, and $f^{(k)}$ for $k \ge 1$ is iteratively defined as
\begin{equation}\label{E:iter-CE}
\begin{aligned}
f^{(k)}(x) = \inf\{ & \lambda f^{(k-1)}(x_1) + (1 - \lambda) f^{(k-1)}(x_2): \\ 
& \lambda \in [0,1], x_1, x_2 \in \mRd, \lambda x_1 + (1-\lambda) x_2 = x\},
\end{aligned}
\end{equation}
then the convex envelope $u = f^{(d)}$ by Carath\'{e}odory's theorem. Consequently, at the continuous level this process terminates in at most $d$ iterations. The method in \cite{Bartels} is a discrete version of this iteration on a structured grid $h\mZd$ with interpolation on the finer grid $h^2 \mZd$, namely $x\in h\mZd$ but $x_1,x_2\in h^2\mZd$ in \eqref{E:iter-CE}. This is thus a two-scale method, with coarse scale $h$, but conceptually different from ours because it does not solve a PDE but rather an algebraic iteration. Moreover, it assumes $u = f$ in a layer $\{x \in \Omega:  \dist(x,\pO) \le Ch \}$ near the boundary $\pO$ to deal with nodes in this region. 

Regarding convergence rates, both the method in \cite{Bartels} and our two-scale method exhibit provable linear rates with respect to the coarse scale for solutions $u\in C^{0,1}(\overline\Omega)$ according to \Cref{R:two-scenarios} (two important scenarios); moreover, \Cref{R:two-scenarios} also shows that our method is quadratic in the coarse scale $\delta$ and linear in the fine scale $h$ for $u \in C^{1,1}(\overline\Omega)$ .
Performing $d$ iterations of the discrete version of \eqref{E:iter-CE} is enough for linear convergence, whereas those for Howard's method cannot be quantified a priori. However, practice reveals that $10$ iterations of Howard's method are enough for convergence, which is consistent with its superlinear structure. Our iterations are simpler than those in \cite{Bartels} because they require much fewer interpolation points. Finally, our two-scale method is designed to work on unstructured meshes and deal with the Dirichlet boundary condition in a natural fashion. The boundary layer effect is handled via discrete barrier functions.
 
\section*{Acknowledgement}
We are grateful to Dimitrios Ntogkas for allowing us to modify his codes on the two-scale method for the \MA equation to solve the convex envelope problems.

\bibliographystyle{amsplain}

\end{document}